\documentclass[10pt,leqno,oneside]{amsart}
\usepackage{stmaryrd}
\usepackage{amsmath}
\usepackage{amssymb}
\usepackage{amstext}
\usepackage{indentfirst}
\usepackage{cite}
\usepackage{mathrsfs}
\usepackage{enumerate}
\usepackage{amsfonts}
\usepackage{enumitem}
\usepackage{marvosym}
\usepackage{color}
\usepackage{pifont}

\newtheorem{Thm}{Theorem}
\newtheorem{Prop}{Proposition}
\newtheorem{Cor}{Corollary}
\newtheorem{Lem}{Lemma}

\newtheorem{Rmk}{Remark}
\newtheorem{Exa}{Example}
\newtheorem{Cla}{Claim}
\newtheorem{Que}{Question}

\title{Closed vacuum static spaces with closed conformal vector fields}
\author{Jian Ye}
\date{\today}

\address{State Key Laboratory of Mathematical Sciences, Academy of Mathematics and Systems Science, Chinese Academy of Sciences, Beijing 100190, China}
\email{yejian@amss.ac.cn}

\begin{document}

\let\thefootnote\relax
\footnotetext{MSC2020: Primary 53C21, Secondary 53C25.}

\maketitle

 \begin{abstract}
       This article aims to classify closed vacuum static spaces with a non-Killing closed conformal vector field. We firstly provide several characterizations of the conditions under which the first derivative of the warping function fulfills the vacuum static equation. Then we establish an identity involving the characteristic function of a conformal vector field on a Riemannian manifold. As applications, we derive a rigidity theorem on closed Riemannian manifolds with a non-Killing closed conformal vector field under suitable conditions and classify closed vacuum static spaces admitting such a vector field.
 \end{abstract}


\section{Introduction}

          The study of conformal vector fields and vacuum static spaces has garnered significant attention in recent years because of their profound implications in geometry and physics. Conformal vector fields are important concept in both Riemannian and pseudo-Riemannian geometry. Meanwhile, vacuum static spaces are crucial objects that naturally arise in the context of both Riemannian geometry and general relativity. This paper aims to explore the interplay between these two concepts, shedding light on their mutual influences.
   \subsection{Vacuum static spaces}
          To begin with, we recall the definition of vacuum static spaces. Let $(M^n, g)$ be a smooth Riemannian manifold with dimension $n\ge3.$ Given a nonconstant smooth function $f$ defined on $M,$ we say that the triple $(M^n,g, f)$ is a vacuum static space if the metric $g$ and $f$ satisfy
          \begin{equation}\label{vss}
                \text{Hess}_{g}f-\Delta_{g} fg=f\text{Ric}_{g},
            \end{equation}
          where $\text{Hess}_{g}f,\Delta_{g},\text{Ric}_{g}$ respectively denote the Hessian of $f,$ the Beltrami-Laplacian, the Ricci curvature tensor of the metric $g$. The equation (\ref{vss}) is called the vacuum static equation.

          We now examine the foundational characterizations of vacuum static spaces. First of all, in \cite{FM1975Duke}, Fischer and Marsden established that the $L^2-$formal adjoint of the linearization of the scalar curvature is given by
          $$\mathcal{L}_{ g}^{*}(f)=\text{Hess}_{g}f-\Delta_{g} fg-f\operatorname{Ric}_{g}.$$
          In particular, the triple $(M^n,g, f)$ is a vacuum static space if and only if $f$ lies in the kernel of the operator $\mathcal{L}_{ g}^{*}.$ Secondly, it is well-known that $(M^n, g)$ with a positive function $f$ constitutes a vacuum static space if and only if the Lorentzian manifold $(M^n\times\mathbb R, g-f^2dt^2)$ satisfies the Einstein equation for a vacuum spacetime (cf. \cite[Proposition 2.7]{Cor2000CMP}). For more information, one can refer to \cite{KobayashiObata1981} \cite{Beig1992} \cite{Cor2000CMP} and references therein. Therefore, it is crucial to explore and classify vacuum static spaces. 

          Until now, numerous examples of vacuum static spaces have been discovered. Firstly, space forms serve as typical examples of vacuum static spaces. Additional nontrivial locally conformally flat examples were constructed by Kobayashi \cite{Kobayashi1982} and Lafontaine\cite{Lafontaine1983}. In this article, we present intriguing examples of warped product vacuum static spaces with nonzero Cotton tensor(cf. Example \ref{basicex}). These examples are not locally conformally flat and thus were not addressed in \cite{Kobayashi1982} or \cite{Lafontaine1983}.

          The existence of nonconstant smooth solution to (\ref{vss}) imposes strong restrictions on the background metric $g.$ For instance, the scalar curvature of complete vacuum static space is constant \cite{FM1975Duke} or \cite[Proposition 2.3]{Cor2000CMP} and the level set $f^{-1}(0)$ is a totally geodesic regular hypersurface \cite{Bourguignon1975} or \cite[Proposition 2.6]{Cor2000CMP}. Besides, at each point $p\in M^n\backslash f^{-1}(0),$ there is a coordinate chart around $p$ such that both $f$ and the metric $g$ are analytic \cite[Proposition 2.8]{Cor2000CMP}.

          Substantial efforts have been dedicated to investigating the classifications of vacuum static spaces. Among the recent progress on this subject, we remember that Kobayashi \cite{Kobayashi1982} and Lafontaine \cite{Lafontaine1983} independently showed that $(n\ge3)-$dimensional closed locally conformally flat vacuum static spaces with scalar curvature $n(n-1)$ are isometric to either the standard unit sphere $\mathbb S^n(1)$ or a finite quotient of $\mathbb S^1(\sqrt{1/n})\times\mathbb S^{n-1}(\sqrt{(n-2)/n})$ or a finite quotient of the warped product $\mathbb S^1(\sqrt{1/n})$ $\times_h$ $\mathbb S^{n-1}$ $(\sqrt{(n-2)/n}).$ Here $h$ is a positive periodic function defined on the sphere $\mathbb S^1(\sqrt{1/n})$ with radius $\sqrt{1/n}.$ Meanwhile, the classification of complete Einstein vacuum static spaces can be derived from the results in \cite[Theorem A]{Obata1962}, \cite[Theorem B]{Kanai1983}, and \cite[Corollary E]{Kanai1983}. Later, motivated by the work of Cao and Chen \cite{CaoChen2012,CaoChen2013}, Qing and Yuan in \cite{QingYuan2013} proved that closed Bach-flat vacuum static spaces with scalar curvature $n(n-1)$ are isometric to either $\mathbb S^n(1)$ or a finite quotient of $\mathbb S^1(\sqrt{1/n})\times F$ or a finite quotient of $\mathbb S^1(\sqrt{1/n})$ $\times_h$ $F.$ Here $F$ means an $(n-1)-$dimensional closed Einstein manifold with Einstein constant $n.$ Recently, inspired by Kim's method used in \cite{Kim2017}, Kim-Shin \cite{KimShin2018} classified the four-dimensional vacuum static spaces with harmonic curvature. Recall that we say that $(M, g)$ has harmonic curvature if the divergence of the Riemann curvature vanishes. In the case of higher dimensions, one can consult \cite{Li2021} or \cite{Kim}. In recent years, Baltazar, Barros, Batista, and Viana\cite{BBBV2020}, as well as the author\cite{Jian2023}, have demonstrated that closed vacuum static spaces of dimension \(n\geq 5\) with zero radial Weyl curvature are Bach-flat. The classification of these spaces is achieved through the results presented in\cite{QingYuan2013}. For more progress on the rigidity, we refer the reader to \cite{Shen1997} \cite{BessieresLR2000} \cite{Lafontaine2009} \cite{Ambrozio2017} \cite{Baltazar2018} \cite{KimShin2018} \cite{HwangYun2021JGA} \cite{HwangYun2024} and references therein.
   \subsection{Conformal vector fields}
          Before proceeding, we review conformal vector fields on a Riemannian manifold $(M^n, g).$ As an extension of Killing vector fields, conformal vector fields have been extensively researched in both Riemannian and pseudo-Riemannian geometry. As vacuum static spaces, the presence of a non-Killing conformal vector field imposes stringent constraints on the underlying Riemannian manifold. Numerous studies indicate that Riemannian manifolds admitting such a vector field are quite rare. For instance, as shown in \cite{Yano1952}, the negativity of the Ricci curvature of a manifold rules out the existence of any conformal vector field on any closed Riemannian manifold. On the other hand, it is known that any compact Einstein manifold with dimension at least three admits a non-Killing conformal vector field if and only if it is isometric to a round sphere(cf.\cite{Obata1962}). Consequently, any conformal vector field on the real projective space of dimension at least three is Killing.

          A long-standing folklore conjecture regarding the presence of conformal vector fields has been widely discussed: the Euclidean sphere is believed to be the unique closed manifold with constant scalar curvature that supports a non-Killing conformal vector field(see\cite{Ejiri}).
          However, this conjecture does not hold in general. In fact, in \cite{Ejiri}, Ejiri constructed a counterexample by identifying a positive function \(h\) on the unit circle \(\mathbb{S}^1(1)\) such that the warped product \(\mathbb{S}^1(1)\times_h N\), where \(N\) is an \((n-1)\)-dimensional compact Riemannian manifold with positive constant scalar curvature, supports a non-Killing closed conformal vector field. In \cite[Remark 2]{Derdzinski1980}, an explicit counterexample is provided by $\mathbb S^1(1)\times_h\mathbb S^3(1),$ where $h=\sqrt{2+\sin t}$.

          Although the aforementioned conjecture does not hold in general, it has been proven in various forms under additional assumptions. Specifically, Nagano \cite{Nagano1959} established the conjecture when the Ricci tensor is parallel. In \cite{1966}, Bishop and Goldberg showed that the conjecture was true in two-dimensional settings. Tanno and Weber \cite[Theorem 1]{TannoWeber1969} established the conjecture when the conformal vector field was closed and vanished at some point. When the conformal vector field is a gradient, the conjecture holds true(cf.\cite{YanoObata1965} or \cite[Theorem 4.4]{Xu1993}). Additionally, when the norm of the Ricci curvature is constant, the conjecture is also valid(see \cite{Lichnerowicz1964} or \cite[Theorem 2]{XuYe}). In particular, any conformal vector field on the product $\mathbb{S}^k(r_1)\times\mathbb{S}^{n-k}(r_2)$($k\in\{1,\cdots,n-1\}$) is Killing. We refer the readers to \cite{1966}\cite{DAS2014CM} \cite{DAS2014}\cite{DAS2012}\cite{DAS2008}\cite{TannoWeber1969}\cite{Xu1993}\cite{Yano1975}\cite{YanoObata1965}\cite{Nagano1959} and references therein for more results on this subject.

          Therefore, a classification of Riemannian manifold admitting a non-Killing conformal vector field is a challenge. In this article, we consider some special cases. Namely, we focus on classifying closed vacuum static spaces that carry a non-Killing closed conformal vector field.
    \subsection{The connection between vacuum static spaces and conformal vector fields}
          Substantial evidences suggest a profound connection between conformal vector fields and vacuum static spaces. On one hand, certain vacuum static spaces naturally encode non-Killing conformal vector fields. For example, the space forms carry non-Killing closed conformal vector fields. In addition to space forms, let $N$ be an $(n-1)-$dimensional closed Einstein manifold with Einstein constant $n,$ the warped product $\mathbb S^1(\sqrt{1/n})\times_h N$ is a vacuum static space with scalar curvature $n(n-1)$ which admits a non-Killing closed conformal vector field $h\frac{\partial}{\partial t}.$

          On the other hand, the characteristic function(for its definition, refer to Subsection 2.2 below) associated with a conformal vector field serves as a solution to the vacuum static equation on Riemannian manifolds with certain geometric conditions. For instance, on a Riemannian manifold with harmonic curvature, the characteristic function of a closed conformal vector field fulfills vacuum static equation(see the proof of Proposition D4 in \cite{Lafontaine1983}). Based on this observation, in \cite[Proposition D4]{Lafontaine1983}, Lafontaine proved that each closed manifold with scalar curvature $n(n-1)$ and harmonic curvature, which admits a non-Killing closed conformal vector field, is isometric to $\mathbb S^n(1)$ or a finite quotient of $\mathbb S^1(\sqrt{1/n})$ $\times_h$ $N.$ 
          Meanwhile, on an Einstein manifold, the characteristic function also satisfies vacuum static equation(cf.\cite[Lemma 2.2]{Herzlich2016}). For the variational characterization of the characteristic function being a solution to vacuum static equation, one may refer to \cite[Theorem 1.1]{MT2017}. In addition, on a Riemannian manifold with constant scalar curvature, the characteristic function of a conformal vector field also satisfies (\ref{Deltaf}).

          Motivated by the foregoing studies, for a Riemannian manifold $(M^n, g)$ with constant scalar curvature which admits a non-homothetic conformal vector field $\xi$ on $(M^n, g)$ with the characteristic function $f,$ two questions come up naturally now:
          \begin{itemize}
            \item[1.] When $f$ fulfills vacuum static equation?
            \item[2.] What are vacuum static spaces that carry a non-Killing conformal vector field?
          \end{itemize}

          In this article, regarding the first question, based on an identity involving $f,$ we will provide a necessary and sufficient condition for \(f\) to satisfy the vacuum static equation. As for the second question, as an application of the identity, we will give a classification of closed vacuum static spaces with a non-Killing closed conformal vector field.
   \subsection{Main results}
          We are now in a position to state our main results. It is well-known that warped product vacuum static spaces with a one-dimensional base were constructed by Kobayashi in \cite{Kobayashi1982} or Lafontaine in \cite{Lafontaine1983}. Particularly, if $(N , \bar{g})$ is an Einstein manifold, it is known from \cite[Lemma C3]{Lafontaine1983} that $\dot{h}$ is automatically a solution to (\ref{vss}) on the warped product $\mathbb{S}^1\times_h N.$

          It is natural to explore the impact of \(\dot{h}\) being a solution to (\ref{vss}) on both \((N ,\bar{g})\) and the warped product \(I\times_h N \). For instance, one might investigate whether one can deduce that $N$ is Einstein from the condition that $\dot{h}$ satisfies (\ref{vss}). To this end, we introduce the following fundamental identity:
    \begin{Prop}\label{warpedproduct}
        Let $h$ be a positive function on an open interval $I$ in $\mathbb{R}$ and $N$ be an $(n-1)-$dimensional Riemannian manifold. Suppose that the metric $g$ of the warped product $I\times_hN$ has constant scalar curvature and $\xi=h\frac{\partial}{\partial t}.$ Then
        $$\mathcal{L}_g^{\star}f(\cdot,\cdot)=-\rm{C}(\cdot,\xi,\cdot),$$
        where $f=\dot{h}$ and $\rm{C}$ denotes the Cotton tensor of $I\times_hN.$
    \end{Prop}

         As a straightforward consequence of Proposition \ref{warpedproduct}, we have

    \begin{Cor} Under the assumption of Proposition {\rm\ref{warpedproduct}}, suppose that $n\geq4,$ if the set $\{\dot{h}\neq0\}$ is dense in $I\times_hN,$ then the following five conditions are equivalent:
        \begin{enumerate}[leftmargin=0.6cm, itemindent=0.4cm]
          \item[(i).] The function $\dot{h}$ is a nonconstant smooth solution to {\rm(\ref{vss})} on $I\times_hN$;
          \item[(ii).] The $(0,2)-$tensor $\rm{C}(\cdot,\xi,\cdot)$ vanishes identically;
           \item[(iii).] The Cotton tensor $\rm{C}$ of $I\times_hN$ vanishes identically;
          \item[(iv).] The warped product $I\times_hN$ has harmonic curvature;
          \item[(v).] The Riemannian manifold $N$ is Einstein.
        \end{enumerate}
    \end{Cor}

         Note that $h\frac{\partial}{\partial t}$ is a conformal vector field on the warped product $I\times_hN$ with characteristic function \(\dot{h}\). Proposition \ref{warpedproduct} offers an identity involving $\dot{h}.$ This suggests that a similar formula might hold for the characteristic function of a conformal vector field on general Riemannian manifolds.

         Additionally, the following identity provides further insight. According to the sixth line from the bottom on page 69 in {\rm\cite{Lafontaine1983}}, the characteristic function \(f\) of a closed conformal vector field on a Riemannian manifold with harmonic curvature satisfies the following identity:
        \[\mathcal{L}_g^{\star}f(\cdot,\cdot)=0.\]

         Our primary objective is to establish an analogous identity for the characteristic function of a conformal vector field on a general Riemannian manifold. More precisely, we have
    \begin{Thm}\label{firstthm}
          Let $(M, g)$ be a Riemannian manifold and $\xi$ be a conformal vector field on $(M, g)$ with characteristic function $f.$ Then, there exists a symmetric $(0,2)-$tensor $\rm{\Phi}$ such that
          $$\mathcal{L}_g^{\star}f(\cdot,\cdot)=\rm{\Phi}(\cdot,\cdot).$$
          In particular, the function $f$ satisfies $(\ref{vss})$ at some point in $M$ if and only if $\rm{\Phi}$ vanishes at the same point.
     \end{Thm}

          For the definition of the tensor $\rm{\Phi},$ one can refer to (\ref{Phi}). It is worth noting that $\rm{\Phi}$ depends not only on $\xi$ but also on the Cotton tensor $\rm{C}.$ Under the assumptions of Theorem \ref{firstthm}, when the metric $g$ has constant scalar curvature and $\xi$ is a closed, the tensor $\rm{\Phi}$ is reduced to $-\rm{C}(\cdot,\xi,\cdot)$. Moreover, if \((M,g)\) has harmonic curvature, \(\rm{\Phi}\) vanishes identically. Consequently, the formula in Theorem \ref{firstthm} is indeed a generalization of the two aforementioned identities.

          As an application of Theorem \ref{firstthm}, we establish a rigidity result for closed Riemannian manifolds that satisfy certain curvature conditions and admit a non-Killing closed conformal vector field. Specifically, we have
    \begin{Thm}\label{Cxizero}
          Let $(M^n, g)$ be a closed Riemannian manifold of dimension $n$ with scalar curvature $n(n-1).$ Suppose that $\xi$ is a non-Killing closed conformal vector field on $(M^n, g)$ satisfying
          \[\rm{C}(\cdot,\xi,\cdot)=0,\text{on}, \text{M}^{\text{n}}.\]
          Then the following statements hold:
          \begin{itemize}[leftmargin=0.5cm, itemindent=0.4cm]
         \item[{\rm(1)}.] For $n=3,$ $(M^3, g)$ is isometric to either the unit sphere $\mathbb S^3(1)$ or a finite quotient of the warped product $\mathbb S^1(\sqrt{1/3})\times_h \mathbb{S}^{2}(\sqrt{1/3}).$
         \item[{\rm(2)}.] For $n=4,$ $(M^4, g)$ is isometric to either the unit sphere $\mathbb S^4(1)$ or a finite quotient of the warped product $\mathbb S^1(1/2)\times_h \mathbb{S}^{3}(\sqrt{1/2}).$
         \item[{\rm(3)}.] For $n\geq5,$ $(M^n, g)$ is isometric to either $\mathbb S^n(1)$ or a finite quotient of the warped product $\mathbb S^1(\sqrt{1/n})\times_h F$ where $F$ stands for an $(n-1)-$dimensional closed Einstein manifold with Einstein constant $n.$
       \end{itemize}
    \end{Thm}

          We now introduce the rigidity of closed vacuum static spaces with a non-Killing closed conformal vector field. It is worth pointing out that similar classifications have been carried out for other types of spaces. For instance, in\cite{DRF2017}, the authors explored gradient Ricci solitons on complete Riemannian manifolds that admit a nonparallel closed conformal vector field. In \cite{Filho2020}, the author examined the critical point equation metric admits a non-trivial closed conformal vector field. For classifications of \(m\)-quasi Einstein manifolds, one can refer to\cite{PSS2023RM}, and for quasi-Einstein manifolds, see \cite{Filho2024DG}.

          As the another application of Theorem \ref{firstthm}, we have
    \begin{Thm}\label{CSS1}
          Let $(M^n, g)$ be an $(n\ge3)-$dimensional closed Riemannian manifold with scalar curvature $n(n-1)$ that admits a non-Killing closed conformal vector field. Suppose that $f$ is a nonconstant smooth solution to
              \begin{equation}\label{MainEab}
                    \nabla^2f+\frac{{\rm{R}}}{n(n-1)}fg=(f+a){\rm E}+bg,
              \end{equation}
          Here, $a,b$ are two constants, and ${\rm R}$ and ${\rm E}$ denote the scalar curvature and the trace-free part of the Ricci curvature tensor respectively. Then $(M^n, g)$ is isometric to one of the spaces listed in Theorem {\rm\ref{Cxizero}}.
    \end{Thm}

          Under the assumption of Theorem \ref{CSS1}, it is interesting to note that $f+a$ satisfies vacuum static equation when $b+\frac{\operatorname R}{n(n-1)}a=0.$

          We briefly outline the proof strategy for Theorem \ref{CSS1}. The initial step is to establish that the quantity $(f+a){\rm C}$ vanishes when the closed conformal vector field and the gradient of $f$ are linearly dependent(see Proposition \ref{ldCTzero}). However, these two vector fields may not always be linearly dependent(cf. (3) in Remark \ref{nlin}). This kind of phenomenon makes the classification subtle. To address this, we employ a proof by contradiction to demonstrate that the Cotton tensor of \((M^n,g)\) vanishes when \(b+\frac{\operatorname{R}}{n(n-1)}a\neq 0\). If \(b+\frac{\operatorname{R}}{n(n-1)}a=0\), the method of proof by contradiction remains valid for $n=3$. For $n\ge4,$ the analyticity deduces that the Cotton tensor equals zero. Ultimately, the desired rigidity result is achieved through Theorem \ref{firstthm} and \cite[Proposition D4]{Lafontaine1983} or Theorem \ref{Cxizero}.

\subsection{Outline of the article}
            The structure of this article is organized as follows. In Sect. 2, we will review fundamental facts concerning conformal vector fields, vacuum static spaces and related critical spaces. In Sect. 3, we will provide multiple characterizations of when the first derivative of the warping function serves as a solution to the vacuum static equation. In addition, several intriguing examples of warped product vacuum static spaces with nonzero Cotton tensor will be given. In Sect. 4, an identity for the characteristic function of a conformal vector field will be derived. As an application, Theorem \ref{Cxizero} is proved. Sect. 5 is dedicated to proving Theorem \ref{CSS1} when $b+\frac{\operatorname R}{n(n-1)}a\neq0.$ Sect. 6 aims to classify vacuum static spaces with a non-trivial closed conformal vector field.

  \section{Preliminaries}
           In this section, we introduce some fundamental concepts in Riemannian geometry and review several well-known results concerning conformal vector fields, vacuum static spaces, and related critical spaces.
    \subsection{Notations and conventions}
        We start with the set-up of notations and reviewing some basic notions in Riemannian geometry.  See  {\cite{Bess1987} \cite{GHS2004} \cite{Petersen2016}} as references on Riemannian geometry.

        Let $(M^n, g)$ be a smooth Riemannian manifold of dimension $n\ge3$ with Levi-Civita connection $\nabla$ induced by the metric $g.$ Denote by $TM$ the tangent bundle of $M^n$ and by $\Gamma(TM)$ the space of smooth sections of $TM.$ The Riemann curvature tensor $R:\Gamma(TM)\times \Gamma(TM)\times \Gamma(TM)\rightarrow \Gamma(TM)$ is defined by
        $$R(X, Y) Z=\nabla_{X} \nabla_{Y} Z-\nabla_{Y} \nabla_{X} Z-\nabla_{[X, Y]} Z,$$
        for any smooth vector fields $X, Y,Z\in\Gamma(TM).$

        Throughout the article, the Einstein convention on summing over the repeated indices will be used. Given a local coordinate system $\left\{\frac{\partial}{\partial x^i}\right\}_{i=1}^n,$ the metric $g$ can be locally represented as:
        \[g=g_{ij}dx^i\otimes dx^j.\]
        In the sequel, the inverse of the matrix $(g_{ij})_{i,j=1}^n$ is denoted by  $(g^{ij})_{i,j=1}^n.$

        The Riemannian metric $g$ induces norms on all the tensor bundles. Precisely, the squared norm of a $(r,s)$-tensor field $\mathcal{U}$ in the coordinate system $\left\{\frac{\partial}{\partial x^i}\right\}_{i=1}^n$ is given by
        $$\Vert \mathcal{U}\Vert^2=g_{i_1k_1}\cdots g_{i_sk_s}g^{j_1l_1}\cdots g^{j_rl_r}\mathcal{U}_{j_1\cdots j_r}^{i_1\cdots i_s}\mathcal{U}_{l_1\cdots l_r}^{k_1\cdots k_s},$$
        where $\mathcal{U}_{j_1\cdots j_s}^{i_1\cdots i_r}$ are components of $\mathcal{U}$ in the coordinate system $\left\{\frac{\partial}{\partial x^i}\right\}_{i=1}^n.$

        We now write the Riemann curvature tensor in the coordinate system $\left\{\frac{\partial}{\partial x^i}\right\}_{i=1}^n$. The components of the $(1,3)-$curvature tensor is given by
        \[R_{i j k}^{l} \frac{\partial}{\partial x^{l}}= {R}(\frac{\partial}{\partial x^{j}}, \frac{\partial}{\partial x^k}) \frac{\partial}{\partial x^{i}}\]
        and the associated $(0,4)-$version by
        \[R_{i j k l}=g_{i s} R_{jkl}^{s}.\]

        We next review several classical curvature tensors. The Ricci curvature tensor $ {\rm Ric}$ of $g$ is obtained through the contraction of the Riemann curvature tensor, which is expressed as
        \[R_{ij}=g^{kl} R_{ikjl}.\]
        The scalar curvature \(\text{R}\) is then defined as the contraction of the Ricci tensor, given by
        \[\text{R}=g^{ij}R_{ij}.\]
        The Schouten tensor $ {\rm A}$ of $g$ is defined by
        \begin{equation*}
              {\rm A}={\rm Ric}-\frac{{\rm R}}{2(n-1)}\,g.
        \end{equation*}
        The Weyl tensor $\rm{W}$ is given by
                \begin{equation*}\label{decompofR}
                        W_{ijkl}=R_{i j k l} - \frac{1}{n-2}({\rm A}\owedge  g)_{ijkl},
                \end{equation*}
        where the symbol $\owedge$ represents the Kulkarni-Nomizu product which is defined for any two symmetric $(0,2)-$tensors $U$ and $V$ as follows
        $$(U \owedge V)_{i j k l}=U_{ik}V_{jl}+U_{jl}V_{ik}-U_{il}V_{jk}-U_{jk}V_{il}.$$
        Through the Schouten tensor, we can define the Cotton tensor $ {\rm C}$ as follows
                \begin{eqnarray}\label{Cotton}
                        C_{ijk}=A_{ij,k}-A_{ik,j}.
                \end{eqnarray}
         It is easy to check that
                         \begin{equation}\label{skewofCotton}
                         C_{ijk}+ C_{ikj}=0 ~\text{and}~ C_{ijk}+C_{jki}+C_{kij}=0,
                        \end{equation}
         and it is also trace-free in any two indices, i.e.
                \begin{equation}\label{vofC}C_{iik}=0=C_{iji}.\end{equation}
         As a straightforward consequence of the second equality in (\ref{skewofCotton}), we have
                \begin{equation}\label{halfC}
                  C_{kij}C_{ijk}=-\frac{1}{2}\|{\rm C}\|^2.
                \end{equation}

         We now recall the left interior multiplication with respect to a vector field. Given a vector field $\xi\in\Gamma(TM)$ and a $(0,s)-$tensor $\mathcal{T},$ the left interior multiplication $i_{\xi}{\mathcal{T}}$ is defined by
          $$i_{\xi}{\mathcal{T}}(X_1,\cdots,X_s):={\mathcal{T}}(\xi,X_1,\cdots,X_s), ~\forall X_1,\cdots,X_s\in\Gamma(TM).$$

         We conclude this subsection with the definition of Hessian and Laplace-Beltrami operator. Given a real $C^2$ function $f$ defined on $M,$ the Hessian $\nabla^2f$ of $f$ is given by
         \begin{equation}\label{Hessian}
            \nabla^2f(X,Y)=X(Y(f))-\nabla_XY(f), ~\forall X, Y\in\Gamma(TM).
         \end{equation}
         The Laplace-Beltrami operator $\Delta $ is defined to be
         \[\Delta f=\text{tr}~\nabla^2f.\]

    \subsection{Conformal vector fields}

         In this subsection, we review the definition of conformal vector fields and collect several well-known established results related to them. For more information, one can refer to {\cite{DuggalSharma1999} \cite{SD2024} \cite{Yano1975} and references therein.

         Let $(M^n ,g)$ be a smooth Riemannian manifold of dimension $n.$ A smooth vector field $\xi$ on $M$ is said to be conformal if there exists a smooth function $f:M\rightarrow{\mathbb R}$ satisfying
         \begin{equation}\label{cvf}\mathcal L_{\xi}g=2f g,\end{equation}
         where $\mathcal L_{\xi}$ is the Lie derivative with respect to $\xi.$ The function $f$ is referred to as the characteristic function of $\xi.$ When $f$ is constant, then $\xi$ is called a homothetic vector field. In particular, if $f$ vanishes identically, $\xi$ is said to be a Killing vector field. We say that $\xi$ is non-Killing if $f$ is not identically zero.

         A conformal vector field $\xi$ is closed if its dual $1-$form $\xi^{\flat}$ is. The dual $1-$form $\xi^{\flat}$ of $\xi$ is defined by $\xi^{\flat}(X)=g(\xi,X)$ for any vector field $ X\in\Gamma(TM).$ If a conformal vector field is the gradient of a smooth function, then it is called a gradient conformal vector field.

         It is well-known that a vector field is conformal if and only if the flow generated by it consists of conformal transformations. In addition, the expression (\ref{cvf}) is equivalent to
         \begin{equation}\label{nablaXnablaY} g(\nabla_X\xi,Y)+g(X,\nabla_Y\xi)=2fg(X,Y),\end{equation}
         for any vector fields $X, Y\in\Gamma(TM).$ By taking trace on both sides of (\ref{nablaXnablaY}), we immediately get
         \begin{equation*}\label{cffdiv}f=\frac{ {{\rm div}\,\xi}}{n}.\end{equation*}

         We now compute the covariant derivative of a conformal vector field $\xi.$ Define a skew symmetric tensor field $\mathcal{P}$ on $M$ by
         \begin{equation}\label{skew}2g(\mathcal{P}(X),Y)=d\xi^{\flat}(X,Y).\end{equation}
         By virtue of Koszul's formula {(cf.\cite[page 54]{Petersen2016})}, we have
         $$2g(\nabla_X\xi,Y)=\mathcal L_{\xi}g(X,Y)+d\xi^{\flat}(X,Y).$$
         Substituting (\ref{cvf}) and (\ref{skew}) into previous formula, it follows that
         \begin{equation}\label{dfa}\nabla_X\xi=fX+\mathcal{P}(X). \end{equation}
         Making use of (\ref{dfa}), a direct computation reveals that the covariant derivative of $\mathcal{P}$ takes the following form
         \begin{equation}\label{CovariantofP}
         \nabla_X\mathcal{P}(Y)=R(X,\xi)Y+Y(f)X-g(X,Y)\nabla f,
         \end{equation}
         which leads to
         \begin{equation}\label{divofP}
         -\sum_{i=1}^n\nabla_{e_{i}}\mathcal{P}(e_{i})=\text{Ric}(\xi)+(n-1)\nabla f,
         \end{equation}
         where $\{e_1, e_2, \cdots, e_n\}$ is a local orthonormal frame on $(M^n, g).$ If $\xi$ is closed, the above two expressions become
         \begin{eqnarray}\label{R}
               R(X,\xi)Y+Y(f)X-g(X,Y)\nabla f&=&0,\\\label{ricandf}
         \operatorname{Ric}(\xi)+(n-1) \nabla f&=&0.
         \end{eqnarray}
         These two formulas will be used in the proof of Theorem \ref{Cxizero}, Theorem \ref{CSS}, and Theorem \ref{VSSCVF}. In addition, by $(ii)$ in Proposition 2 of \cite{Kerbrat1976}, it is well-known that the every zero point of nonzero conformal vector field is isolated on a completed connected Riemannian manifold. In the sequel, we will use this fact explicitly throughout the paper.

         We conclude this subsection by providing a well-known classification of closed conformal vector fields based on their covariant derivatives.
         \begin{Lem}\label{CCVF}
                Let $\xi$ be a conformal vector field on a Riemannian manifold $(M, g).$ Then $\xi$ is closed if and only if there exists a function $f:M\rightarrow\mathbb{R},$ such that, for any vector field $X\in\Gamma(TM),$ it holds that
                \begin{equation}\label{nablaXi}\nabla_X\xi=fX.\end{equation}
         \end{Lem}
         \begin{proof}
               If $\xi$ is closed, the desired formula follows from (\ref{dfa}). Conversely, assume that there exists a function $f:M\rightarrow\mathbb{R}$ such that (\ref{nablaXi}) holds. As a result, the formula (\ref{nablaXnablaY}) is valid. Equivalently, $\xi$ is a conformal vector field. The closedness of $\xi$ follows from (\ref{nablaXi}) and the basic formula
                $$d\xi^{\flat}(X,Y)=g(\nabla_X\xi,Y)-g(X,\nabla_Y\xi).$$
         \end{proof}

    \subsection{Vacuum static spaces and related critical spaces}
              In this subsection, we collect several facts on vacuum static spaces and related critical spaces for the subsequent study. Let $(M^n, g, f)$ be a vacuum static space. By taking the trace in both sides of (\ref{vss}), we have
         \begin{equation}\label{Deltaf}
         \Delta f+\frac{\operatorname R}{n-1}f=0.
         \end{equation}
              Substitute (\ref{Deltaf}) into the initial equation (\ref{vss}) to rewrite it locally as
         \begin{equation}\label{MainEVSS}f_{ij}+\frac{\operatorname R}{n(n-1)}fg_{ij}=fE_{ij},\end{equation}
              where $E_{ij}=R_{ij}-\frac{\operatorname R}{n}g_{ij}.$
              Without loss of generality, we next consider the equation (\ref{MainEab}). By taking the trace in both sides of (\ref{MainEab}), we have
         \begin{equation}\label{Deltafab}
               \Delta f+\frac{\operatorname R}{n-1}f=nb.
         \end{equation}
              As \cite[Proposition 2.3]{Cor2000CMP}, it can be proven that the scalar curvature of \((M^n,g)\) is constant when \(f\) is a non-constant solution of (\ref{MainEab}) on \(M^n\).
              Associated to (\ref{MainEab}), there is a covariant $(0,3)-$tensor ${\rm T}={\rm T}^f$ which can be written as
         \begin{equation}\label{T}
                 \begin{aligned}
                        T_{ijk}=&\frac{n-1}{n-2}\left(E_{ik}f_{j}-E_{ij}f_{k}\right)+\frac{1}{n-2}\left(g_{ik}E_{jl}-g_{ij}E_{kl}\right)f_{l}.
            \end{aligned}\end{equation}
               A simple observation reveals that the tensor ${\rm T}$ shares the same symmetries as the Cotton tensor. In fact, by definition of $\operatorname{T},$ there holds
               \begin{equation}\label{Tijk=-Tikj}T_{ijk}=-T_{ikj} \quad \text {and} \quad T_{ijk}+T_{jki}+T_{kij}=0.\end{equation}
               Besides, one can check directly that ${\rm T}$ is trace-free in any two indices, i.e.
                \begin{equation}\label{Tiji=0}T_{iik}=0=T_{iji}.\end{equation}

                We close this section by stating three basic lemmas.

        \begin{Lem}\label{decomposeofR}
              Let $f$ be a smooth solution to {\rm(\ref{MainEab})} defined on an $(n\ge3)-$dimensional Riemannian manifold. Then the following identities hold true:
                         \begin{eqnarray}\label{RlijkflVSS}
                         &R_{ijk}^lf_{l}=\big(E_{ij}-\frac{\operatorname R}{n(n-1)}g_{ij}\big)f_{k}-\big(E_{ik}-\frac{\operatorname R}{n(n-1)}g_{ik}\big)f_{j}+(f+a)C_{ijk},&\\\label{fCVSS}
                         &(f+a)C_{ijk}=(i_{\nabla f}W)_{ijk}+T_{ijk}.&
                         \end{eqnarray}
        \end{Lem}

        \begin{Lem}\label{formforTfE}
              Let $f$ be a smooth solution to {\rm(\ref{MainEab})}. Then
             $$E_{ik}T_{ijk}f_{j}=\frac{n-2}{2(n-1)}\|{\rm T}\|^2.$$
        \end{Lem}

        \begin{Lem}\label{scalarPosi}
               Let $f$ be a nonconstant smooth solution to {\rm(\ref{MainEab})} defined on an $(n\ge3)-$dimensional closed Riemannian manifold. Then the scalar curvature ${\rm R}$ is positive.
        \end{Lem}
        \begin{proof}
               Let $u:=f+a.$ We rewrite (\ref{Deltafab}) in terms of $u$ as follows:
               \begin{equation}\label{Deltau}
                 \Delta u+\frac{\operatorname R}{n-1}u=n\kappa,
               \end{equation}
               where $\kappa=b+\frac{\operatorname R}{n(n-1)}a.$ By multiplying both sides of (\ref{Deltau}) by $u$ and integrating over $M,$ the integration by parts gives rise to
               \[\frac{\operatorname R}{n-1}\int_Mu^2-n\kappa\int_Mu=\int_M\|\nabla f\|^2.\]
               On the other hand, by integrating both sides of (\ref{Deltau}) over $M,$ we have
               \begin{equation}\label{Inu}
                 \frac{\operatorname R}{n-1}\int_Mu=n\kappa Vol(M,g),
               \end{equation}
               where the symbol $Vol(M,g)$ stands for the volume of $(M, g).$ Combine the previous two formulas to deduce that
               \[\frac{\operatorname R}{n-1}\int_M(u-\bar{u})^2=\int_M\|\nabla f\|^2,\]
               where $\bar{u}$ is defined by
               $$\bar{u}:=\frac{1}{Vol(M,g)}\int_Mu.$$
               Therefore, the scalar curvature ${\rm R}$ is positive unless $f$ is constant.
        \end{proof}

\section{Warped product vacuum static spaces}

             In this section, we explore warped product vacuum static spaces with a one-dimensional base. It is well-known that such spaces were constructed by Kobayashi in \cite[Lemma 1.1]{Kobayashi1982} or Lafontaine in \cite[Theorem C1]{Lafontaine1983}. To elaborate on their construction, we need to introduce some notation.

             Let $(I, dt\otimes dt)$ be an open interval in $\mathbb{R}$ endowed with the standard metric induced from $\mathbb{R},$ $(N, \bar{g})$ a Riemannian manifold of dimension $n-1,$ and  $h$ a smooth positive function on $I.$ Consider the warped product
             $$I\times_hN =(I\times N, g=dt\otimes dt+h(t)^2\bar{g}).$$
             The function $h$ is referred to as the warping function, while \(N\) is called the fiber. In what follows, the first and second derivatives of $h(t)$ with respect to $t\in I$ will be denoted by $\dot{h}$ and $\ddot{h}$, respectively.

             We now state Kobayashi or Lafontaine's construction. If $h$ is a positive constant, we can assume without loss of generality that \(h\equiv 1\). According to the statements in lines 15-18 on page 65 of \cite{Lafontaine1983}, if $(N, \bar{g})$ is an $(n-1)-$dimensional Einstein manifold with positive scalar curvature ${\rm{R}},$ and $(\mathbb{S}^1\times N, dt\otimes dt+\bar{g}, f(t,\cdot))$ constitutes a vacuum static space, then $f$ depends only on the parameter $t\in \mathbb{S}^1$ and satisfies the differential equation
             \begin{equation}\label{ddotf}
               \ddot{f}+\frac{{\rm{R}}}{n-1}f=0, \text{on~} \mathbb{S}^1.
             \end{equation}
             Actually, by Corollary \ref{INRP} below, if the scalar curvature of $(N, \bar{g})$ is a nonzero constant, then $(I\times N, dt\otimes dt+\bar{g}, f(t,\cdot))$ forms a vacuum static space if and only if $(N , \bar{g})$ is Einstein and $f$ is defined on $I$ satisfying (\ref{ddotf}).  \\
             If $h$ is not constant, in \cite[Lemma 1.1]{Kobayashi1982}, Kobayashi proved that, if $(N , \bar{g})$ has constant sectional curvature, then $(I\times_hN , f=f(t))$ is a vacuum static space if and only if $f=c\dot{h}$ for some nonzero constant $c$ and there exists a constant $c_1$ such that, the warping function $h$ satisfies
     \begin{equation}\label{warpedvss}
             \ddot{h}+\frac{{\rm{R}}}{n(n-1)}h=c_1h^{1-n},
     \end{equation}
             where ${\rm{R}}$ denotes the scalar curvature of $I\times_hN .$ It is worth highlighting that (\ref{warpedvss}) holds only if the scalar curvature of $I\times_h N $ is constant (to see the item $(ii)$ in Lemma \ref{IN} below). Moreover, the existence of nonconstant, smooth, positive periodic solution to (\ref{warpedvss}) on $\mathbb{R}$ was demonstrated under suitable conditions in \cite[Proposition 2.6]{Kobayashi1982}. Furthermore, if the condition on the fiber $(N, \bar{g})$ is relaxed to being an Einstein manifold, it is known from \cite[Lemma C3]{Lafontaine1983} that $\dot{h}$ is automatically a solution to (\ref{vss}) on $\mathbb{S}^1\times_h N,$ and the kernel of the linear operator $\mathcal{L}_g^{\star}$ is spanned by the singleton $\{\dot{h}\}.$

             It is natural to explore the impact of \(\dot{h}\) being a solution to (\ref{vss}) on both the fiber \((N ,\bar{g})\) and the warped product \(I\times_h N \). For instance, one might ask whether $\dot{h}$ being a solution to (\ref{vss}) on $I\times_h N $ implies that $(N , \bar{g})$ is Einstein. We will provide an affirmative answer at the end of this section.

             To proceed, we continue to introduce some notation. We denote by $\nabla,\overline{\nabla}$ the Levi-Civita connection of the metric $g,\bar{g},$ respectively.

             We now collect several elementary formulas for the warped product $I\times_hN $ as follows.
    \begin{Lem}{\rm{(\cite{Bess1987})}}\label{LCofw}
             Let $X, Y\in\Gamma(TN ).$ The Levi-Civita connection of the metric $g$ is given by
          \begin{eqnarray}\label{geo}
          \nabla_{\frac{\partial}{\partial t}}\frac{\partial}{\partial t}  &=& 0, \\\label{partialX}
          \nabla_{\frac{\partial}{\partial t}}X &=&\frac{\dot{h}}{h}X=\nabla_{X}\frac{\partial}{\partial t}, \\\label{partialXY}
          \nabla_{X}Y &=&\overline{\nabla}_XY-\frac{\dot{h}}{h}g(X,Y)\frac{\partial}{\partial t}.
        \end{eqnarray}
    \end{Lem}
    \begin{Lem}{\rm{(\cite[Proposition 9.106]{Bess1987})}}\label{Ricw}
            Let $X, Y\in\Gamma(TN )$. The Ricci curvature ${\rm{Ric}}$ of $I\times_h N $ satisfies
        \begin{eqnarray}\label{Ric1}
          {\rm{Ric}}(\frac{\partial}{\partial t},\frac{\partial}{\partial t})  &=& -(n-1)\frac{\ddot{h}}{h}, \\\label{Ric2}
          {\rm{Ric}}(X,\frac{\partial}{\partial t}) &=&0, \\\label{Ric3}
          {\rm{Ric}}(X,Y) &=& \overline{{\rm{Ric}}}(X,Y)-\frac{1}{h^2}\left[h\ddot{h}+(n-2)\dot{h}^2\right]g(X,Y),
        \end{eqnarray}
            where $\overline{{\rm{Ric}}}$ denotes the Ricci curvature of the metric $\bar{g}.$
    \end{Lem}
    \begin{Lem}{\rm{(\cite{Bess1987})}}
            The scalar curvature ${\rm{R}}$ of $I\times_h N $ is
        \begin{equation}\label{RbarR}
          {\rm{R}}h^2=\overline{{\rm{R}}}-(n-1)(n-2)\dot{h}^2-2(n-1)h\ddot{h},
        \end{equation}
            where $\overline{{\rm{R}}}$ stands for the scalar curvature of the metric $\bar{g}.$
    \end{Lem}

            We now give some remarks.
    \begin{Rmk}\label{dddothddoth}
            \begin{itemize}[leftmargin=0.5cm, itemindent=0.4cm]
              \item[{\rm(1)}.] It is interesting to mention that Ejiri in {\rm\cite[Lemma 3.3]{Ejiri}} verified the existence of nonconstant, positive periodic solution to {\rm(\ref{RbarR})} on $\mathbb{R}$ when both ${\rm{R}}$ and $\overline{{\rm{R}}}$ are positive constants.
              \item[{\rm(2)}.] It is easy to see that $\overline{{\rm{R}}}$ is constant if ${\rm{R}}$ is constant due to {\rm(\ref{RbarR})}.
              \item[{\rm(3)}.] Substituting {\rm(\ref{RbarR})} into {\rm(\ref{Ric3})}, we obtain
        \[{\rm{Ric}}(X,Y)=\mathring{\overline{{\rm{Ric}}}}(X,Y)+\Big(\frac{\ddot{h}}{h}+\frac{{\rm{R}}}{n-1}\Big)g(X,Y),\forall X,Y\in\Gamma(TN ),\]
        where $\mathring{\overline{{\rm{Ric}}}}$ denotes the trace-free part of $\overline{{\rm{Ric}}}.$
            \end{itemize}

    \end{Rmk}

            We now collect several differential equations involving the warping function $h$ that are related to (\ref{RbarR}) and have been discussed in \cite{Kobayashi1982} or \cite{Lafontaine1983}. Specifically, we have
    \begin{Lem}\label{IN}
            Suppose that the scalar curvature of the warped product $I\times_h N $ is constant. Then the following assertions hold:
            \begin{itemize}[leftmargin=0.6cm, itemindent=0.4cm]
              \item[(i).]  The warping function $h$ satisfies\footnote{It is worth pointing out that this formula is precisely the formula $(1.11)$ in \cite{Kobayashi1982}.}
                           $$\frac{d^3h}{dt^3}+(n-1)\frac{\dot{h}\ddot{h}}{h}+\frac{{\rm{R}}}{n-1}\dot{h}=0, ~\forall t\in I,$$
                           where $\frac{d^3h}{dt^3}$ denotes the third derivative of $h(t)$ with respect to $t\in I.$
              \item[(ii).] The differential equation {\rm(\ref{warpedvss})} holds for some constant $c_1.$
              \item[(iii).] There exists a constant $\tau$ such that the warping function $h$ fulfils
                          \[\dot{h}^2+\frac{{\rm{R}}}{n(n-1)}h^2-\frac{\overline{{\rm{R}}}}{(n-1)(n-2)}=\tau h^{2-n}, ~\forall t\in I.\]
            \end{itemize}
            Furthermore, the constants $c_1$ and $\tau$ satisfy $\tau+\frac{2c_1}{n-2}=0.$
    \end{Lem}
    \begin{proof}
           As mentioned in $(2)$ of Remark \ref{dddothddoth}, the scalar curvature $\overline{{\rm{R}}}$ of $(N, \bar{g})$ is constant due to (\ref{RbarR}) and the constancy of ${\rm R}.$ As a result, the item $(i)$ follows from differentiating both sides of (\ref{RbarR}).\\
           For the item $(ii),$ as in the fourth line from the bottom on page 668 of \cite{Kobayashi1982}, a simple computation shows that
           \[\frac{d}{dt}\Big\{h^{n-1}\big[\ddot{h}+\frac{{\rm{R}}}{n(n-1)}h\big]\Big\}=h^{n-1}\Big[\frac{d^3h}{dt^3}+(n-1)\frac{\dot{h}\ddot{h}}{h}+\frac{{\rm{R}}}{n-1}\dot{h}\Big].\]
           The desired equality is obtained by observing that the right hand side of previous formula vanishes owing to the equation in the item $(i)$.\\
           For the item $(iii),$ it is not hard to check that
           \begin{eqnarray*}
           &&\frac{d}{dt}\Big\{h^{n-2}\big[\dot{h}^2+\frac{{\rm{R}}}{n(n-1)}h^2-\frac{\overline{{\rm{R}}}}{(n-1)(n-2)}\big]\Big\}\\
              &=&h^{n-3}\dot{h}\Big[2h\ddot{h}+(n-2)\dot{h}^2+\frac{{\rm{R}}}{n-1}{h}^2-\frac{\overline{{\rm{R}}}}{n-1}\Big].
           \end{eqnarray*}
           The desired identity is derived from noting that the right hand side of previous formula vanishes, which boils down to (\ref{RbarR}).

           The relation between $c_1$ and $\tau$ is derived by eliminating $\ddot{h}$ using equations (\ref{RbarR}) and (\ref{warpedvss}), and then comparing the result with the formula established in $(iii).$ According to this relation, one can check that the formula in the item $(iii)$ is identical to the formula $(1.6)$ in \cite{Kobayashi1982}. The proof of Lemma \ref{IN} is finished.
    \end{proof}

           We now recall a classical result on $I\times_hN.$ It is well-known that the vector field $h\frac{\partial}{\partial t}$ is a conformal vector field(cf. for instance, \cite[Lemma 2.4]{Ejiri} or \cite[Lemma 2.2]{Brendle2013}). In fact, it is also closed. More precisely, we have
    \begin{Lem}\label{CCVFhpartial}
          The vector field $\xi=h\frac{\partial}{\partial t}$ is a closed conformal vector field on $I\times_hN$ with characteristic function $\dot{h}.$
    \end{Lem}
    \begin{proof}
          For the convenience of readers, we provide a proof here. By means of Lemma \ref{CCVF}, it suffices to prove that, for any $X\in\Gamma(T(I\times_hN)),$
          $$\nabla_X\xi=\dot{h}X.$$
          To see this, assume that $X=X^1\frac{\partial}{\partial t}+X^2$ where $X^1=g(X,\frac{\partial}{\partial t})$ and $X^2$ is the projection of $X$ onto the slice $\{t\}\times N ,$ by (\ref{geo}) and (\ref{partialX}), together with the definition of $\xi,$ we have
          \begin{eqnarray*}
          \nabla_X\xi&=& X^1\nabla_{\frac{\partial}{\partial t}}\xi+\nabla_{X^2}\xi \\
             &=&X^1\dot{h}\frac{\partial}{\partial t}+\dot{h}X^2 \\
             &=&\dot{h}X.
          \end{eqnarray*}
          This finishes the proof of Lemma \ref{CCVFhpartial}.
    \end{proof}
    \begin{Rmk}\label{fformulaR}
         \begin{itemize}[leftmargin=0.5cm, itemindent=0.4cm]
           \item[{\rm(1)}.] For the characterization of conformal vector fields on the warped product $I\times_hN ,$ one can consult {\rm\cite[Appendix A]{JaureguiWylie2015}}.
           \item[{\rm(2)}.] Under the assumption of Lemma {\rm\ref{IN}}, the item $(i)$ in Lemma {\rm\ref{IN}} can also be deduced from the following well-known identity:
         \begin{equation*}
              \Delta f+\frac{\rm{R}}{n-1}f+\frac{1}{2(n-1)}\xi({\rm{R}})=0,
        \end{equation*}
        where $f$ is the characteristic function of a conformal vector field $\xi$ on a Riemannian manifold $(M^n, g).$ One can refer to {\rm\cite[Lemma 4.2]{Xu1993}} or $(3.5)$ on page $160$ in {\rm\cite{Yano1975}} for the proof of previous identity. Indeed, by taking $\xi=h\frac{\partial}{\partial t},$ the desired formula follows from the fact that the characteristic function of $\xi$ equals $\dot{h},$ {\rm(\ref{DeltafRwp})}, and the constancy of ${\rm{R}}.$
         \end{itemize}
    \end{Rmk}

          As can be seen from the above, the warped product \(I\times_h N\) inherently possesses excellent properties. We next present two additional intriguing properties of \(I\times_h N\) with constant scalar curvature. First of all, we now prove that the $(0,2)-$tensor $i_{\frac{\partial}{\partial t}}{\rm{C}}$ vanishes on such a space. More precisely, we have the following result:
    \begin{Prop}\label{iCzero}
          Assume that $I\times_hN $ has constant scalar curvature. Then, at each point on $I\times_hN $, the following identity holds:
          $$i_{\frac{\partial}{\partial t}}{\rm{C}}(\cdot,\cdot)=0.$$
    \end{Prop}
    \begin{proof}
          Firstly, since the Cotton tensor is skew symmetric with respect to the last two indices, $i_{\frac{\partial}{\partial t}}{\rm{C}}$ is skew-symmetric. Thus
          \begin{equation*}\label{C1}
                i_{\frac{\partial}{\partial t}}{\rm{C}}(\frac{\partial}{\partial t},\frac{\partial}{\partial t})=0.
          \end{equation*}
          To complete the proof, it remains to show that for any two vector fields $X,Y\in\Gamma(TN ),$
          \begin{eqnarray}\label{C2}
            i_{\frac{\partial}{\partial t}}{\rm{C}}(\frac{\partial}{\partial t},X)&=& 0, \\\label{C3}
            i_{\frac{\partial}{\partial t}}{\rm{C}}(X,Y) &=& 0.
          \end{eqnarray}
          For (\ref{C2}), by the definitions of $i_{\xi}{\rm{C}}$ and the Cotton tensor (\ref{Cotton}), together with the fact that the metric $g$ has constant scalar curvature, it is not hard to check that
          \begin{eqnarray*}
          i_{\frac{\partial}{\partial t}}{\rm{C}}(\frac{\partial}{\partial t},X)&=& {\rm{C}}(\frac{\partial}{\partial t},\frac{\partial}{\partial t},X) \\
             &=&\nabla_X{\rm{A}}(\frac{\partial}{\partial t},\frac{\partial}{\partial t})-\nabla_{\frac{\partial}{\partial t}}{\rm{A}}(\frac{\partial}{\partial t},X)\\
             &=&-2{\rm{A}}(\frac{\partial}{\partial t},\nabla_X{\frac{\partial}{\partial t}})+{\rm{A}}(\nabla_{\frac{\partial}{\partial t}}\frac{\partial}{\partial t},X)+{\rm{A}}(\frac{\partial}{\partial t},\nabla_{\frac{\partial}{\partial t}}X)
          \end{eqnarray*}
          where the last equality results from the definition of covariant derivative and (\ref{Ric1}),(\ref{Ric2}) established in Lemma \ref{Ricw}. Together with (\ref{geo}) and (\ref{partialX}), we obtain
          \begin{eqnarray*}
          i_{\frac{\partial}{\partial t}}{\rm{C}}(\frac{\partial}{\partial t},X)&=&-\frac{\dot{h}}{h}{\rm{A}}(\frac{\partial}{\partial t},X)\\
          &=&0,
          \end{eqnarray*}
          where the second equality comes from (\ref{Ric2}).

          We now verify that (\ref{C3}) holds true. Repeat similar process in the proof of (\ref{C2}) to obtain
          \begin{eqnarray*}
          &&i_{\frac{\partial}{\partial t}}{\rm{C}}(X,Y)\\
             &=&\nabla_{Y}{\rm{A}}(\frac{\partial}{\partial t},X)-\nabla_X{\rm{A}}(\frac{\partial}{\partial t},Y)\\
             &=&-{\rm{A}}(\nabla_{Y}\frac{\partial}{\partial t},X)-{\rm{A}}(\frac{\partial}{\partial t},\nabla_{Y}X)+{\rm{A}}(\nabla_{X}\frac{\partial}{\partial t},Y)+{\rm{A}}(\frac{\partial}{\partial t},\nabla_{X}Y)\\
             &=&{\rm{A}}(\frac{\partial}{\partial t},[X,Y]),
          \end{eqnarray*}
          where the last equality is due to (\ref{partialX}) and the fact that the metric $g$ is of torsion free. The desired equality (\ref{C3}) follows from combining previous formula with the relation $[X,Y]\in\Gamma(TN )$ and (\ref{Ric2}). The proof of Proposition \ref{iCzero} is completed.
    \end{proof}
    \begin{Rmk}\label{symC}
          In the situation of Proposition {\rm\ref{iCzero}}, by means of the first Bianchi identity of the Cotton tensor, it is easy to see that the $(0,2)-$tensor ${\rm{C}}(\cdot,\frac{\partial}{\partial t},\cdot)$ is symmetric.
    \end{Rmk}

          To further our analysis, we provide the explicit expression for the action of the operator $\mathcal{L}_g^{\star}$ on functions defined on $I\times_hN$. Specifically, we have
    \begin{Lem}\label{Lgh}
           On the warped product $I\times_hN,$ for any $f\in C^\infty(I\times_hN)$ and vector fields $X, Y\in\Gamma(TN),$ the following relations hold:
           \begin{eqnarray}\label{L1}
           \mathcal{L}_g^{\star}f(\frac{\partial}{\partial t},\frac{\partial}{\partial t}) &=&\frac{n-1}{h}\left(\ddot{h}f-\dot{h}\frac{\partial f}{\partial t}\right)-\frac{1}{h^2}\overline{\Delta}f, \\\label{L2}
           \mathcal{L}_g^{\star}f(X,\frac{\partial}{\partial t}) &=&X(\frac{\partial f}{\partial t}-\frac{\dot{h}}{h}f), \\\label{L3}
           \mathcal{L}_g^{\star}f(X,Y) &=&\overline{\nabla}^2 f(X,Y)-f\mathring{\overline{{\rm{Ric}}}}(X,Y)\\\nonumber
           &&-\Big[\Delta f+\frac{{\rm{R}}}{n-1}f+\frac{1}{h}\big(\ddot{h}f-\dot{h}\frac{\partial f}{\partial t}\big)\Big]g(X,Y),
         \end{eqnarray}
         where $\overline{\Delta}f,\overline{\nabla}^2 f$ denotes the Laplace-Beltrami operator, the Hessian of $f$ with respect to the metric $\bar{g}.$ In particular,
         \begin{equation}\label{Ldoth}
           -\mathcal{L}_g^{\star}\dot{h}=\dot{h}\mathring{\overline{{\rm{Ric}}}}+h^2\Big[\frac{d^3}{dt^3}h+(n-1)\frac{\dot{h}\ddot{h}}{h}+\frac{{\rm{R}}}{n-1}\dot{h}\Big]\bar{g}.
         \end{equation}
    \end{Lem}
    \begin{proof}
          We begin with the computation of $\Delta f.$ Making use of (\ref{partialXY}), a straightforward computation indicates that
         \begin{equation}\label{DeltafRwp}
           \Delta f=\frac{\partial^2f}{\partial t^2}+(n-1)\frac{\dot{h}}{h}\frac{\partial f}{\partial t}+\frac{1}{h^2}\overline{\Delta}f,
         \end{equation}
         where $\Delta$ stands for the Laplace-Beltrami operator of the metric $g.$\\
         For (\ref{L1}), a direct calculation yields
         \begin{eqnarray*}
          \mathcal{L}_g^{\star}f(\frac{\partial}{\partial t},\frac{\partial}{\partial t}) &=& \nabla^2f(\frac{\partial}{\partial t},\frac{\partial}{\partial t})-\Delta f-f\text{Ric}(\frac{\partial}{\partial t},\frac{\partial}{\partial t}) \\
            &=&\frac{\partial^2f}{\partial t^2}+(n-1)\frac{\ddot{h}}{h}f-\Delta f,
         \end{eqnarray*}
         where the second equality is derived from (\ref{Hessian}), (\ref{geo}), and (\ref{Ric1}). In conjunction with (\ref{DeltafRwp}), we know that (\ref{L1}) is valid. \\
         For (\ref{L2}), in light of (\ref{Ric2}), we can see that
         \begin{eqnarray*}
          \mathcal{L}_g^{\star}f(X,\frac{\partial}{\partial t}) &=& \nabla^2f(X,\frac{\partial}{\partial t}) \\
            &=&X(\frac{\partial f}{\partial t})-\nabla_X\frac{\partial}{\partial t}(f)\\
            &=&X(\frac{\partial f}{\partial t})-\frac{\dot{h}}{h}X(f),
         \end{eqnarray*}
         where the second equality comes from (\ref{Hessian}) and the last equality from (\ref{partialX}).\\
         For (\ref{L3}), it is not difficult to check that
         \begin{eqnarray*}
         \mathcal{L}_g^{\star}f(X,Y)&=&\nabla^2f(X,Y)-\Delta fg(X,Y)-f\text{Ric}(X,Y) \\
            &=&\overline{\nabla}^2 f(X,Y)+\Big[\frac{\dot{h}}{h}\frac{\partial f}{\partial t}-\Delta f\Big]g(X,Y)\\
            &&-f\mathring{\overline{{\rm{Ric}}}}(X,Y)-f\Big(\frac{\ddot{h}}{h}+\frac{{\rm{R}}}{n-1}\Big)g(X,Y),
         \end{eqnarray*}
         where the second equality follows from (\ref{Hessian}), (\ref{partialXY}), and the item $(3)$ in Remark \ref{dddothddoth}. The desired expression (\ref{L3}) is derived by rearranging the preceding formula.\\
         If $f=\dot{h},$ the formula (\ref{Ldoth}) is obtained by putting (\ref{L1}), (\ref{L2}), (\ref{L3}), and (\ref{DeltafRwp}) together. This completes the proof of Lemma \ref{Lgh}.
    \end{proof}

         It is interesting to note that Lemma \ref{Lgh} does not require the scalar curvature of $I\times_hN $ to be constant. As straightforward consequences of Lemma \ref{Lgh}, we have
    \begin{Cor}\label{INRP}
         Suppose that the scalar curvature of $(N, \bar{g})$ is nonzero constant. Then the triple $(I\times N, dt\otimes dt+\bar{g}, f(t,\cdot))$ forms a vacuum static space if and only if $(N, \bar{g})$ is Einstein and $f$ is defined on $I$ satisfying {\rm(\ref{ddotf})}.
    \end{Cor}
    \begin{proof}
         From (\ref{L1}), (\ref{L2}), (\ref{L3}), and (\ref{DeltafRwp}), it is not difficult to check that the necessity holds. We now turn our attention to the sufficiency. If $ f(t,\cdot))$ satisfies {\rm(\ref{MainEVSS})}, by {\rm(\ref{L2})}, we obtain
         $$X(\frac{\partial f}{\partial t})=0, \text{for any}~ X\in\Gamma(TN). $$
         This implies that $f$ can be decomposed as $f(t,\cdot)=f_1(t)+f_2(\cdot).$ Thanks to {\rm(\ref{Deltaf})} and (\ref{DeltafRwp}), we arrive at
         \begin{equation}\label{f1f2IN}
           \ddot{f_1}(t)+\frac{{\rm R}}{n-1}f_1(t)+\overline{\Delta}f_2+\frac{{\rm R}}{n-1}f_2=0.
         \end{equation}
         On the other hand, with the aid of {\rm(\ref{Deltaf})}, the formula {\rm(\ref{L3})} reduces to
         \begin{equation}\label{fRic}
           0=\mathcal{L}_g^{\star}f(X,Y)=\overline{\nabla}^2 f_2(X,Y)-f\mathring{\overline{{\rm{Ric}}}}(X,Y).
         \end{equation}
         By taking trace on both sides of {\rm(\ref{fRic})}, we have
         \begin{equation*}
           \overline{\Delta}f_2=0.
         \end{equation*}
         Substituting previous formula into {\rm(\ref{f1f2IN})}, one can see that
         \[\ddot{f_1}(t)+\frac{{\rm R}}{n-1}f_1(t)+\frac{{\rm R}}{n-1}f_2=0.\]
         Since \(\mathrm{R}\) is a nonzero constant, it follows that \(f_2\) must be constant. Therefore, $f$ is defined on $I$ and satisfies {\rm(\ref{ddotf})}. Finally, from {\rm(\ref{fRic})}, we conclude that $(N,\bar{g})$ is Einstein.
    \end{proof}
    \begin{Cor}\label{warpedPonh}
         Assume that $h$ is not constant on $I.$ Let $f\in C^\infty(I).$ If $(I\times_h N, f(t))$ is a vacuum static space, then it follows that $f=c_2\dot{h}$ for some constant $c_2.$
    \end{Cor}
    \begin{proof}
         In fact, from {\rm(\ref{L1})}, we have
         \[\ddot{h}f=\dot{h}\dot{f}.\]
         This can be rearranged to yield
         \[\frac{d}{dt}\Big(\frac{f}{\dot{h}}\Big)=0.\]
         Consequently, \(\frac{f}{\dot{h}}\) is a constant, which implies the desired result.
    \end{proof}

          To investigate the influence of the fact that $\dot{h}$ is a solution to (\ref{vss}) on the fiber $(N, \bar{g}),$ we introduce the following fundamental identity:

    \begin{Prop}\label{warpedproduct3}
          Suppose that the metric $g$ of $I\times_hN $ has constant scalar curvature and $\xi=h\frac{\partial}{\partial t}.$ Then, for any vector fields $X,Y\in\Gamma(T(I\times_hN)),$ the following identity holds:
        $$\mathcal{L}_g^{\star}\dot{h}(X,Y)=-{\rm{C}}(X,\xi,Y).$$
    \end{Prop}
    \begin{proof}
          Let $f=\dot{h}.$ By virtue of (\ref{L1}), (\ref{L2}), Remark \ref{symC}, and Proposition \ref{iCzero}, it is easy to see that, for any vector field $X\in\Gamma(TN ),$ the following relations hold
        \begin{eqnarray*}
           \mathcal{L}_g^{\star}f(\frac{\partial}{\partial t},\frac{\partial}{\partial t}) &=& 0=-{\rm{C}}(\frac{\partial}{\partial t},\xi,\frac{\partial}{\partial t}), \\
           \mathcal{L}_g^{\star}f(X,\frac{\partial}{\partial t}) &=& 0=-{\rm{C}}(\frac{\partial}{\partial t},\xi,X)=-{\rm{C}}(X,\xi,\frac{\partial}{\partial t}).
         \end{eqnarray*}
         It remains to prove that
         \begin{eqnarray}\label{LfCXY}
         \mathcal{L}_g^{\star}f(X,Y) =-{\rm{C}}(X,\xi,Y), \forall X,Y\in\Gamma(TN ).
         \end{eqnarray}
         To see this, by the definition of the Cotton tensor (\ref{Cotton}) and the fact that the metric $g$ has constant scalar curvature, we arrive at
         \begin{eqnarray*}
          {\rm{C}}(X,\xi,Y) &=&\nabla_{Y}{\rm{A}}(X,\xi)-\nabla_{\xi}{\rm{A}}(X,Y)\\
            &=&-{\rm{A}}(\nabla_{Y}X,\xi)-{\rm{A}}(X,\nabla_{Y}\xi)\\
            &&-\xi({\rm{A}}(X,Y))+{\rm{A}}(\nabla_{\xi}X,Y)+{\rm{A}}(X,\nabla_{\xi}Y),
         \end{eqnarray*}
         where the second equality results from the definition of covariant derivative to $(0,2)-$tensor and (\ref{Ric2}). Together with (\ref{partialX}), (\ref{partialXY}), and (\ref{Ric2}), we obtain
         \begin{eqnarray*}
          {\rm{C}}(X,\xi,Y) &=&\dot{h}{\rm{A}}(\frac{\partial}{\partial t},\frac{\partial}{\partial t})g(X,Y)-\xi({\rm{A}}(X,Y))+\dot{h}{\rm{A}}(X,Y)\\
          &=&-(n-1)\frac{\dot{h}\ddot{h}}{h}g(X,Y)-\xi({\rm{Ric}}(X,Y))+f{\rm{Ric}}(X,Y).
         \end{eqnarray*}
         where the second equality comes from (\ref{Ric1}), the constancy of scalar curvature and the definition of the Shouten tensor ${\rm A}.$ By means of the formula constructed in $(3)$ of Remark \ref{dddothddoth}, we conclude
         \begin{eqnarray*}
          {\rm{C}}(X,\xi,Y) &=&f{\rm{Ric}}(X,Y)-\left(\frac{2{\rm{R}}}{n-1}\dot{h}+\frac{\dot{h}\ddot{h}}{h}+\frac{d^3h}{dt^3}+(n-1)\frac{\dot{h}\ddot{h}}{h}\right)g(X,Y)\\
          &=&f\Big[{\rm{Ric}}(X,Y)-\Big(\frac{\ddot{h}}{h}+\frac{{\rm{R}}}{n-1}\Big)g(X,Y)\Big],
         \end{eqnarray*}
         where the second equality is derived from the item $(i)$ in Lemma \ref{IN}. Combining previous formula with $(3)$ of Remark \ref{dddothddoth}, we immediately get
         \begin{equation*}\label{CbarRic}
           {\rm{C}}(X,\xi,Y)=f\mathring{\overline{{\rm{Ric}}}}(X,Y).
         \end{equation*}
         On the other hand, it follows from (\ref{Ldoth}) and the item $(i)$ in Lemma \ref{IN} that
         \begin{equation*}\label{LbarRic}
           -\mathcal{L}_g^{\star}f(X,Y)=f\mathring{\overline{{\rm{Ric}}}}(X,Y).
         \end{equation*}
         Eventually, the desired formula (\ref{LfCXY}) is obtained by combining previous two formulas. This ends the proof of Proposition \ref{warpedproduct3}.
    \end{proof}

         It is interesting to note that $(0,2)-$tensor ${\rm{C}}(\cdot,\xi,\cdot)$ is naturally symmetric owing to Remark \ref{symC}. We now give some remarks about Proposition \ref{warpedproduct3}.
    \begin{Rmk}
         In {\rm\cite[Proposition 2.1]{Brendle2013}}, if
         \begin{itemize}
           \item[{\rm 1)}.] The function $\dot{h}$ is positive on $I,$
           \item[{\rm 2)}.] $\overline{{\rm{Ric}}}\geq (n-2)\rho\bar{g}$ for some constant $\rho,$
           \item[{\rm 3)}.] The function
                      $$2\frac{\ddot{h}}{h}-(n-2)\frac{\rho-\dot{h}^2}{h^2}$$
                      is non-decreasing for $t\in I,$
         \end{itemize}
          Brendle showed that
         $$-\mathcal{L}_g^{\star}\dot{h}\geq0,\text{~on,~} I\times_hN .$$
         It is worthy noting that the formula in Proposition {\rm\ref{warpedproduct3}} relates $\mathcal{L}_g^{\star}\dot{h}$ to the Cotton tensor under the condition that the metric $g$ has constant scalar curvature. When the scalar curvature of $g$ is not constant, a similar formula can still be formulated(cf. $(2)$ in Remark {\rm\ref{closedmathcalLf}} in next section).
    \end{Rmk}

    \begin{Rmk}\label{CXYN}
         In the situation of Proposition {\rm\ref{warpedproduct3}}, from the above proof, we know that
         \[-\mathcal{L}_g^{\star}f(X,Y)=f\mathring{\overline{{\rm{Ric}}}}(X,Y)={\rm{C}}(X,\xi,Y), \forall X,Y\in\Gamma(TN ).\]
    \end{Rmk}

         As a direct result of Proposition \ref{warpedproduct3}, we can conclude that

    \begin{Cor} \label{WPVSS}
         Under the assumption of Proposition {\rm\ref{warpedproduct3}}, suppose that $n\geq4,$ if the set $\{\dot{h}\neq0\}$ is dense in $I\times_hN,$ then the following five statements are equivalent:
        \begin{enumerate}[leftmargin=0.5cm, itemindent=0.4cm]
          \item[(i).] The function $\dot{h}$ is a smooth solution to {\rm(\ref{vss})} on $I\times_hN $;
          \item[(ii).] The $(0,2)-$tensor $\rm{C}(\cdot,\frac{\partial}{\partial t},\cdot)$ vanishes identically on $I\times_hN $;
          \item[(iii).] The Cotton tensor $\rm{C}$ of $I\times_hN $ is identically zero on $I\times_hN $;
          \item[(iv).] The warped product $I\times_hN $ has harmonic curvature;
          \item[(v).] The fiber $(N, \bar{g})$ is Einstein.
        \end{enumerate}
    \end{Cor}
    \begin{proof}
          From the formula established in Proposition \ref{warpedproduct3}, it is evident that the condition $(i)$ is equivalent to $(ii).$ Additionally, due to Remark \ref{CXYN}, the assertions $(ii)$ and $(v)$ are equivalent. Furthermore, by the definition of the Cotton tensor (\ref{Cotton}) and harmonic curvature, together with the metric $g$ has constant scalar curvature, it follows that $(iii)$ is equivalent to $(iv).$ It remains to demonstrate the equivalence between $(ii)$ and $(iii).$ It is clear that $(iii)$ implies $(ii).$ Conversely, assume that the assertion $(ii)$ is true, with the help of Proposition \ref{iCzero}, in order to show that $(iii)$ holds true, it is sufficient to verify that, for $X,Y,Z\in \Gamma(TN ),$
          \begin{equation}\label{CXYZ}
            {\rm{C}}(X,Y,Z)=0,~\text{on},~I\times_hN .
          \end{equation}
          To see this, by the definition of the Cotton tensor (\ref{Cotton}), (\ref{partialXY}), (\ref{Ric2}), (\ref{Ric3}), and the fact that both $g$ and $\bar{g}$ have constant scalar curvatures, a detailed but cumbersome calculation indicates that
          \begin{equation}\label{CXYXNCYZ}
             {\rm{C}}(X,Y,Z)=\overline{{\rm{C}}}(X,Y,Z),
          \end{equation}
          where $\overline{{\rm{C}}}$ stands for the Cotton tensor of $(N , \bar{g}).$ As mentioned before, the condition $(ii)$ implies that $(N , \bar{g})$ is Einstein. As a result, $\overline{{\rm{C}}}$ vanishes identically on $N .$ Finally, (\ref{CXYZ}) is valid due to (\ref{CXYXNCYZ}) and the vanishing of $\overline{{\rm{C}}}.$
          This finishes the proof of Corollary \ref{WPVSS}.
    \end{proof}

    \begin{Rmk}\label{NEin}
          \begin{itemize}[leftmargin=0.3cm, itemindent=0.4cm]
            \item[{\rm(1)}.] The equivalence between the items $(iv)$ and $(v)$ was proved by Derdzi\'nski in {\rm\cite[Lemma 4]{Derdzinski1980}}.
            \item[{\rm(2)}.] As already pointed out, if $(N , \bar{g})$ is Einstein, by {\rm\cite[Lemma C3]{Lafontaine1983}}, it is known that $\dot{h}$ is naturally a solution to {\rm(\ref{vss})}. Combining with Corollary {\rm\ref{WPVSS}}, we conclude that $\dot{h}$ is a solution to {\rm(\ref{vss})} on $I\times_hN $ if and only if $(N, \bar{g})$ is Einstein. This resolves the question posed at the beginning of this section.
            \item[{\rm(3)}.] If the scalar curvature of $I\times_hN$ is not constant, for $X,Y,Z\in \Gamma(TN ),$ a tedious computation reveals that
                             \begin{eqnarray*}
                             &&{\rm C}(\frac{\partial}{\partial t},\frac{\partial}{\partial t},X)=-\frac{1}{2(n-1)}X({\rm R})=-{\rm C}(\frac{\partial}{\partial t},X ,\frac{\partial}{\partial t}),\\
                             &&{\rm C}(\frac{\partial}{\partial t},X,Y)=0,{\rm C}(X,h\frac{\partial}{\partial t},Y)=\dot{h}\mathring{\overline{{\rm{Ric}}}}(X,Y),\\
                             &&{\rm{C}}(X,Y,Z)=\begin{cases}
                                             \frac{1}{4}\big[Z({\rm R})g(X,Y)-Y({\rm R})g(X,Z)\big], & \text{if } n=3, \\
                                             \overline{{\rm{C}}}(X,Y,Z)+Z(\Theta)g(X,Y)-Y(\Theta)g(X,Z), & \text{if } n\geq4.
                                             \end{cases}
                             \end{eqnarray*}
                             where $\Theta=\frac{\overline{{\rm R}}}{2(n-2)}h^{-2}-\frac{{\rm R}}{2(n-1)}.$ Particularly, when \(n\geq 4\), the equivalence between conditions \((ii)\), \((iii)\), and \((v)\) continues to hold true.
            \item[{\rm(4)}.] Under the hypothesis of Corollary {\rm\ref{WPVSS}}, if $(N,\bar{g})$ is not Einstein, then the Cotton tensor of the warped product $I\times_hN$ is not identically zero.
          \end{itemize}
    \end{Rmk}

         We conclude this section with an observation regarding vacuum static spaces on $I\times_hN$. As demonstrated in Corollary \ref{INRP}, if the scalar curvature of $(N, \bar{g})$ is nonzero constant and the triple $(I\times N, dt\otimes dt+\bar{g}, f(t,\cdot))$ forms a vacuum static space, then $f$ is defined on $I.$ Unlike the case on $I\times N,$ on the warped product $I\times_hN,$ the solution to (\ref{MainEVSS}) may depend on the fiber $N.$ To be more precise, we have
    \begin{Cla}\label{propddoth}
         Assume that the scalar curvature $\rm{R}$ of $I\times_hN$ is constant. Let $\bar{f}$ be a nonzero smooth function defined on $N.$ Then $(I\times_hN, h(t)\bar{f}(\cdot))$ is a vacuum static space if and only if $(N, \bar{g}, \bar{f})$ is a vacuum static space and the warping function $h$ satisfies
         \begin{equation}\label{heq}
           \ddot{h}+\frac{\rm{R}}{n(n-1)}h=0,\forall t\in I.
         \end{equation}
    \end{Cla}
    \begin{proof}
         Assume that $f(t,\cdot):=h(t)\bar{f}(\cdot).$ To begin with, we compute $\mathcal{L}_g^{\star}f.$ Using the definition of $f,$ the formula (\ref{L1}) can be rewritten as:
         \begin{eqnarray*}
           \mathcal{L}_g^{\star}f(\frac{\partial}{\partial t},\frac{\partial}{\partial t}) &=&\frac{n-1}{h}\left(h\ddot{h}-\dot{h}^2\right)\bar{f}-\frac{1}{h}\overline{\Delta}\bar{f} \\
            &=&(n-1)\Big[\ddot{h}+\frac{\rm{R}}{n(n-1)}h\Big]\bar{f}-\frac{n-1}{h}\left(\dot{h}^2+\frac{\rm{R}}{n(n-1)}h^2\right)\bar{f}-\frac{1}{h}\overline{\Delta}\bar{f}\\
            &=&(n-1)(c_1-\tau)h^{1-n}\bar{f}-\frac{1}{h}\Big(\overline{\Delta}\bar{f}+\frac{\bar{{\rm R}}}{n-2}\bar{f}\Big),
         \end{eqnarray*}
         where the last equality results from (\ref{warpedvss}) and the item $(iii)$ in Lemma \ref{IN}. Together with the relation $\tau+\frac{2c_1}{n-2}=0$ and (\ref{warpedvss}), we arrive at
         \begin{eqnarray}\label{hbarftt}
          \mathcal{L}_g^{\star}f(\frac{\partial}{\partial t},\frac{\partial}{\partial t}) &=& \frac{n(n-1)}{n-2}c_1h^{1-n}\bar{f}-\frac{1}{h}\Big(\overline{\Delta}\bar{f}+\frac{\bar{{\rm R}}}{n-2}\bar{f}\Big)\nonumber \\
            &=& \frac{n(n-1)}{n-2}\Big[\ddot{h}+\frac{\rm{R}}{n(n-1)}h\Big]\bar{f}-\frac{1}{h}\Big(\overline{\Delta}\bar{f}+\frac{\bar{{\rm R}}}{n-2}\bar{f}\Big).
         \end{eqnarray}
         We now compute the term $\mathcal{L}_g^{\star}f(X,\frac{\partial}{\partial t}).$ It is easy to see that
         $$\frac{\partial f}{\partial t}-\frac{\dot{h}}{h}f=0.$$
         Therefore, the formula (\ref{L2}) turns into
         \begin{equation}\label{hbarftx}
           \mathcal{L}_g^{\star}f(X,\frac{\partial}{\partial t})=0.
         \end{equation}
         To address the term $\mathcal{L}_g^{\star}f(X,Y),$ we first simplify (\ref{DeltafRwp}) as follows:
         \begin{equation}\label{Deltafhbarf}
           \Delta f=\big[\ddot{h}+(n-1)\frac{\dot{h}^2}{h}\big]\bar{f}+\frac{1}{h}\overline{\Delta}\bar{f}.
         \end{equation}
         From the proof of (\ref{L3}) and the relation $f=h\bar{f},$ we have
         \begin{eqnarray*}
         \mathcal{L}_g^{\star}f(X,Y)=\overline{\nabla}^2 f(X,Y)+\Big[\frac{\dot{h}^2}{h}\bar{f}-\Delta f\Big]g(X,Y)-f\text{Ric}(X,Y).
         \end{eqnarray*}
         Inserting (\ref{Deltafhbarf}) into previous formula, we immediately get
         \begin{eqnarray}\nonumber
         \mathcal{L}_g^{\star}f(X,Y)&=&h\overline{\nabla}^2 \bar{f}(X,Y)-h\overline{\Delta}\bar{f}\bar{g}(X,Y)\\\nonumber
         &&-f\Big\{\text{Ric}(X,Y)+\frac{1}{h^2}\big[h\ddot{h}+(n-2)\dot{h}^2\big]g(X,Y)\Big\}\\\label{hbarftxy}
         &=&h\mathcal{L}_{\bar{g}}^{\star}\bar{f}(X,Y),
         \end{eqnarray}
         where the second equality follows from (\ref{Ric3}) and the definition of the operator $\mathcal{L}_{\bar{g}}^{\star}\bar{f}.$ The desired equivalence is obtained by putting (\ref{hbarftt}), (\ref{hbarftx}), (\ref{hbarftxy}), and the equation (\ref{Deltaf}) for $\bar{f}$ together. This ends the proof of Claim \ref{propddoth}.
    \end{proof}
    \begin{Rmk}
       \begin{itemize}[leftmargin=0.3cm, itemindent=0.4cm]
        \item[{\rm 1.}] From the proof of Claim {\rm\ref{propddoth}}, it follows that if $(N, \bar{g},\bar{f})$ is a vacuum static space with zero scalar curvature, then $(I\times N, dr\otimes dr+\bar{g}, \bar{f})$ is also a vacuum static space with zero scalar curvature.
        \item[{\rm 2.}] By {\rm\cite{Bess1987}}, it is well-known that the warped product $I\times_hN$ is Einstein if and only if $(N, \bar{g})$ is Einstein and the warping function $h$ satisfies {\rm(\ref{heq})}.
        \end{itemize}
    \end{Rmk}

        Through Claim \ref{propddoth}, we can construct interesting vacuum static spaces. We begin with basic notation. Given an integer $k\ge2$ and a positive number $r,$ the hyperboloid model of hyperbolic space $\mathbb{H}^k(r)$ is defined as follows:
        \[\mathbb{H}^k(r):=\{x\in\mathbb{R}^n:-x_0^2+x_1^2+\cdots x_k^2=-r^2, x_0>0\}.\]
        The canonical metric of $\mathbb{H}^k(r)$ is denoted by $g_{k, r}.$ Let $f_{k, r}\in\text{Span}\{x_0, x_1, \cdots, x_k\}.$
    \begin{Exa}\label{basicex}
        Let $k\in\{1,2,\cdots,n-3\}$ and $n\geq k+3.$ Consider the warped product manifold $(M^n_k, g_k)$ defined by
        $$\mathbb{R}\times_{\cosh t}\left(\mathbb{H}^k(r_{k})\times\mathbb{H}^{n-k-1}(s_k)\right),$$
        where $r_{k}=\sqrt{k/(n-1)}$ and $s_k=\sqrt{(n-k-2)/(n-1)}.$
        Then the triple $(M^n_k, g_k, $ $\cosh t f_{k,r_{k}})$ is a vacuum static space with scalar curvature $-n(n-1).$
    \end{Exa}

        In fact, it is well-known that the Riemannian product
        $$\left(\mathbb{H}^k(r_{k})\times\mathbb{H}^{n-k-1}(s_{k}), g_{k,r_{k}}+g_{n-k-1, s_{k}}, f_{k,r_{k}}\right)$$
        is an $(n-1)-$dimensional vacuum static space with scalar curvature $-(n-1)(n-2).$ Applying Claim {\rm\ref{propddoth}}, we deduce that $(M^n_k, g_k, \cosh t f_{k,r_{k}})$ is a vacuum static space with scalar curvature $-n(n-1).$

        Additionally, as noted in $(3)$ of Remark \ref{NEin}, for each $k,$ the Cotton tensor of $(M^n_k, g_k, \cosh t f_{k,r_{k}})$ is not identically zero. In particular, $(M^n_k, g_k, \cosh t f_{k,r_{k}})$ is not locally conformally flat and therefore not accounted for in \cite{Lafontaine1983} or\cite{Kobayashi1982}. Moreover, the Ricci curvature tensor of $(M^n_k, g_k, \cosh t f_{k,r_{k}})$ has three distinct eigenvalues. Therefore, the associated $(0,3)-$tensor ${\rm T}$ is also non identically zero.

        As another application of Claim \ref{propddoth}, we show that the kernel of the operator $\mathcal{L}_g^{\star}$ is spanned by a singleton $\{\dot{h}\}$ when $c_1$ in (\ref{warpedvss}) is nonzero. To be more precise, we have
    \begin{Cla}\label{dimKwar}
        Let $N$ be an Einstein manifold. Assume that the scalar curvature $\rm{R}$ of $I\times_hN$ is constant. If $h$ is nonconstant and $c_1\neq0,$ then
        \[{\rm{Ker}}\mathcal{L}_g^{\star}={\rm{Span}}(\{\dot{h}\}).\]
    \end{Cla}
    \begin{proof}
        If the scalar curvature $\rm{R}$ of $I\times_hN$ is constant, it follows from the item $(ii)$ in Lemma \ref{IN} that (\ref{warpedvss}) holds for some constant $c_1$. Note that $c_1$ may be positive, zero, or negative. Meanwhile, when $N$ is Einstein, the equivalence between $(i)$ and $(v)$ in Corollary \ref{WPVSS} implies that $\dot{h}\in{\rm{Ker}}\mathcal{L}_g^{\star}.$ Thus
        \[{\rm{Span}}(\{\dot{h}\})\subset{\rm{Ker}}\mathcal{L}_g^{\star}.\]
        Conversely, suppose that $\mathcal{L}_g^{\star}f=0,$ take advantage of (\ref{L2}) to find
        \[X\left(\frac{\partial}{\partial t}\Big(\frac{f}{h}\Big)\right)=0, \text{for any}~ X\in\Gamma(TN).\]
        This implies that $f$ can be decomposed as
        $$f(t,\cdot)=f_1(t)+h(t)f_2(\cdot),$$
        where $f_2(\cdot)$ is defined on $N.$ With the aid of Corollary \ref{warpedPonh}, it remains to show that $f_2(\cdot)$ is constant on $N.$ We argue by contradiction, supposing that $f_2(\cdot)$ is not constant on $N.$ From (\ref{L1}), it follows that
        \[\frac{n-1}{h}\left(\ddot{h}f-\dot{h}\frac{\partial f}{\partial t}\right)-\frac{1}{h^2}\overline{\Delta}f=0.\]
        In addition, substituting (\ref{DeltafRwp}) into (\ref{Deltaf}) yields
        \[\frac{\partial^2f}{\partial t^2}+(n-1)\frac{\dot{h}}{h}\frac{\partial f}{\partial t}+\frac{1}{h^2}\overline{\Delta}f+\frac{{\rm R}}{n-1}f=0.\]
        The sum of the above two equations produces
        \[\frac{\partial^2f}{\partial t^2}+(n-1)\frac{\ddot{h}}{h}f+\frac{{\rm R}}{n-1}f=0.\]
        Making use of the decomposition $f(t,\cdot)=f_1(t)+h(t)f_2(\cdot),$ we have
        \[\ddot{f}_1(t)+(n-1)\frac{\ddot{h}}{h}f_1(t)+\frac{{\rm R}}{n-1}f_1(t)+\frac{n}{h}\left(\ddot{h}+\frac{{\rm R}}{n(n-1)}h\right)f_2(\cdot)=0.\]
        Since $f_2$ is not constant, it follows that
        \[\ddot{h}+\frac{{\rm R}}{n(n-1)}h=0,~\text{for}~ t\in I.\]
        Combining this with (\ref{warpedvss}), we conclude that the constant \(c_1\) must be zero, which contradicts our assumption $c_1\neq0$. This finishes the proof of Claim \ref{dimKwar}.
    \end{proof}

    \begin{Rmk}
        If $N$ is a closed Einstein manifold, $h$ is not constant, and the scalar curvature of $\mathbb{S}^1\times_hN$ is constant, it is clear that $c_1>0.$ Thus the kernel of the linear operator $\mathcal{L}_g^{\star}$ on $\mathbb{S}^1\times_hN$ is spanned by the singleton $\{\dot{h}\},$ as established in {\rm Lemma C3} of {\rm\cite{Lafontaine1983}}.
    \end{Rmk}

\section{The action of the operator $\mathcal{L}_g^{\star}$ on the characteristic function of a conformal vector field}
    This section is devoted to extending Proposition \ref{iCzero} and Proposition \ref{warpedproduct3} to the context of general Riemannian manifolds endowed with a conformal vector field. Furthermore, an application of Theorem \ref{firstthm} is presented.
  \subsection{Main identity}
    We start by reviewing some notions from Riemannian geometry. Let $(M^n ,g)$ be a Riemannian manifold of dimension $n$ and $\xi\in\Gamma(TM).$ We denote $\xi$ and the dual one-form $\xi^{\flat}$ with respect to a local coordinate system $\left\{\frac{\partial}{\partial x^i}\right\}_{i=1}^n$ by $\xi=\xi^i\frac{\partial}{\partial x^i}$ and $\xi^{\flat}=\xi^{\flat}_{i}dx^i,$ respectively. The differential of $\xi^{\flat}$ is given by
         \[d\xi^{\flat}=-2\Sigma_{j<k}\mathcal{P}_{jk}dx^j\wedge dx^k,\]
         where $\mathcal{P}_{jk}=\frac{1}{2}(\xi^{\flat}_{j,k}-\xi^{\flat}_{k,j}).$ 

         We now generalize Proposition \ref{iCzero} as follows.
    \begin{Prop}\label{iC}
          Let $(M^n, g)$ be a Riemannian manifold of dimension $n$ and $\xi$ be a conformal vector field on $(M^n, g)$. Then
          $$i_{\xi}{\rm{C}}=\frac{1}{2(n-1)}d{\rm{R}}\wedge \xi^{\flat}-\frac{1}{2}d\delta d\xi^{\flat}+{\rm Ric}(d\xi^{\flat}),$$
          where ${\rm Ric}(d\xi^{\flat})=\Sigma_{j<k}(R_{kl}\mathcal{P}_{lj}+\mathcal{P}_{kl}R_{lj})dx^j\wedge dx^k.$
          In particular, if $\xi$ is closed, then
          $$i_{\xi}{\rm{C}}=\frac{1}{2(n-1)}d{\rm{R}}\wedge \xi^{\flat}.$$
    \end{Prop}
    \begin{proof}
          Let $f$ be the characteristic function of $\xi.$ Let $p\in M$ and $\left\{\frac{\partial}{\partial x^i}\right\}_{i=1}^n$ be normal coordinate system centered at $p.$ We firstly rewrite (\ref{dfa}) and (\ref{CovariantofP}) locally as
          \begin{equation}\label{dfa1}
            \xi^{\flat}_{i,j}=fg_{ij}+\mathcal{P}_{ij}
          \end{equation}
          and
          \begin{equation}\label{RxifP}
             R_{ijk}^l\xi_l^{\flat}=g_{ij}f_k-g_{ik}f_j-\mathcal{P}_{jk,i}.
          \end{equation}
          By taking covariant derivative on both sides of (\ref{RxifP}) in the direction of $\frac{\partial}{\partial x^i},$ together with (\ref{dfa1}), we can deduce that
           \begin{equation}\label{Rijki}
             R_{ijk,i}^l\xi_l^{\flat}=-\mathcal{P}_{jk,ii}-R_{ijk}^l\mathcal{P}_{li}.
           \end{equation}
           For the first term in left hand side of (\ref{Rijki}), with the help of the second Bianchi identity of the Riemann curvature tensor, we find
           \begin{eqnarray*}
            R_{ijk,i}^l\xi_l^{\flat} &=&-(R_{iki,j}^l+R_{iij,k}^l)\xi_l^{\flat} \\
              &=&(R_{lj,k}-R_{lk,j})\xi^l.
           \end{eqnarray*}
           Taking into account the definition of the Cotton tensor (\ref{Cotton}), it can be inferred that
           \begin{eqnarray}\label{2}
            R_{ijk,i}^l\xi_l^{\flat} =C_{ljk}\xi^l+\frac{1}{2(n-1)}\big(\nabla_k{\rm{R}} ~\xi_j^{\flat}-\nabla_j{\rm{R}} ~\xi_k^{\flat}\big).
           \end{eqnarray}
           We now address the two terms in right hand side of (\ref{Rijki}). To do that, by virtue of the fact that the two form $d\xi^\flat$ is closed, we have
           \begin{equation*}
             \mathcal P_{jk,i}+\mathcal P_{ki,j}+\mathcal P_{ij,k}=0.
           \end{equation*}
           As a result, a simple computation gives rise to
           \[\mathcal P_{jk,ii}=\mathcal P_{ik,ji}-\mathcal P_{ij,ki}.\]
           Making use of Ricci's identity, we obtain
           \begin{eqnarray*}
           \mathcal P_{jk,ii}&=&\mathcal P_{ik,ij}+R_{jl}\mathcal P_{lk}+R^l_{kji}\mathcal P_{il}-\mathcal P_{ij,ik}-R_{kl}\mathcal P_{lj}-R^l_{jki}\mathcal P_{il} \\
              &=&\mathcal P_{ik,ij}-\mathcal P_{ij,ik}+R_{jl}\mathcal P_{lk}-R_{kl}\mathcal P_{lj}-(R^l_{kji}+R^l_{jik})\mathcal P_{li}.
           \end{eqnarray*}
           The first Bianchi identity for Riemann curvature tensor leads to
           \begin{equation}\label{3}
             -\mathcal{P}_{jk,ii}-R_{ijk}^l\mathcal{P}_{li}=\mathcal P_{ij,ik}-\mathcal P_{ik,ij}+R_{kl}\mathcal P_{lj}+\mathcal P_{kl}R_{lj}
           \end{equation}
            Plugging (\ref{2}) and (\ref{3}) into (\ref{Rijki}), it follows that
            \begin{eqnarray*}
            C_{ljk}\xi^l&=&\frac{1}{2(n-1)}\big({\nabla_j{\rm{R}} ~\xi_k^{\flat}-\nabla_k\rm{R}} ~\xi_j^{\flat}\big)+\mathcal P_{ij,ik}-\mathcal P_{ik,ij}\\
            &&+R_{kl}\mathcal P_{lj}+\mathcal P_{kl}R_{lj}.
            \end{eqnarray*}
           This completes the proof of Proposition \ref{iC}.
    \end{proof}

           It is interesting to point out that the $(0,2)-$tensor $(R_{kl}\mathcal{P}_{lj}+\mathcal{P}_{kl}R_{lj})$ is skew-symmetric due to the symmetry of the Ricci curvature tensor and the skew-symmetry of $\mathcal{P}_{jk}.$ In addition, under the assumptions of Proposition \ref{iC}, if the metric $g$ has constant scalar curvature, then \(i_{\xi}{\rm{C}}=0\). Thus Proposition \ref{iC} is indeed a generalization of Proposition \ref{iCzero}.

           We are now ready to prove Theorem \ref{firstthm}.\\
    {\bf Proof of Theorem \ref{firstthm}.} Let $f$ be characteristic function of a conformal vector field $\xi$ on a Riemannian manifold $(M^n, g)$ of dimension $n$. By taking covariant derivative on both sides of (\ref{RxifP}) in the direction of $\frac{\partial}{\partial x^j},$ together with (\ref{dfa1}), we obtain
          \begin{eqnarray}
             f_{ik}-\Delta_gfg_{ik}-fR_{ik}= R_{ijk,j}^l\xi_l^{\flat}+R_{ijk}^l\mathcal{P}_{lj}+\mathcal{P}_{jk,ij}.
          \end{eqnarray}
          Repeat the procedure in the proof of Proposition \ref{iC} to imply
          \begin{eqnarray*}
          &&f_{ik}-\Delta_gfg_{ik}-fR_{ik}\\
             &=&R_{jkli,j}\xi^l+R_{ijk}^l\mathcal{P}_{lj}+\mathcal{P}_{jk,ji}+R_{il}\mathcal{P}_{lk}+R^l_{kij}\mathcal{P}_{jl}\\
             &=&-(R_{jkij,l}+R_{jkjl,i})\xi^l+R_{ijk}^l\mathcal{P}_{lj}+\mathcal{P}_{jk,ji}+R_{il}\mathcal{P}_{lk}+R^l_{kij}\mathcal{P}_{jl}\\
             &=&(R_{ki,l}-R_{kl,i})\xi^l+\mathcal{P}_{jk,ji}+R_{il}\mathcal{P}_{lk},
          \end{eqnarray*}
          where the last equality follows from the definition of Ricci curvature tensor and the algebraic properties of Riemann curvature tensor. Finally, the definition of the Cotton tensor (\ref{Cotton}) yields
          \[f_{ik}-\Delta_gfg_{ik}-fR_{ik}={\rm\Phi}_{ik},\]
          where
          \begin{equation}\label{Phi}
           {\rm\Phi}_{ik}=-C_{kli}\xi^l-\frac{1}{2(n-1)}\big(\nabla_i{\rm{R}} ~\xi_k^{\flat}-\xi({\rm{R}})g_{ik}\big)+\mathcal{P}_{jk,ji}+R_{il}\mathcal{P}_{lk}.
          \end{equation}
          This completes the proof of Theorem \ref{firstthm}.

    \begin{Rmk}\label{closedmathcalLf}
          \begin{itemize}[leftmargin=0.5cm, itemindent=0.4cm]
            \item[{\rm(1)}.] Relied on Proposition {\rm\ref{iC}} and the second equality in {\rm(\ref{skewofCotton})}, we can conclude that the $(0,2)-$tensor ${\rm\Phi}$ is symmetric.
            \item[{\rm(2)}.] In the situation of Theorem {\rm\ref{firstthm}}, if $\xi$ is closed, then
       \[(\mathcal{L}_g^{\star}f)_{ik}=-C_{kli}\xi^l+\frac{1}{2(n-1)}(\xi({\rm{R}})g_{ik}-\nabla_i{\rm{R}} ~\xi_k^{\flat}).\]
         Moreover, if the metric $g$ has constant scalar curvature, then
          $$(\mathcal{L}_g^{\star}f)_{ik}=-C_{kli}\xi^l.$$
         It is worth noting that when \(\mathrm{C}\equiv 0\), the previous formula reduces to the expression found in the sixth line from the bottom on page $69$ in {\rm\cite{Lafontaine1983}}.
            \item[{\rm(3)}.] By taking trace on both sides of the identity in Theorem {\rm\ref{firstthm}}, together with the relation $\delta^2\xi^\flat=0,$ one can get the formula in the item $(2)$ of Remark {\rm\ref{fformulaR}}.
          \end{itemize}
    \end{Rmk}
   \subsection{Application}
          We now present an application of Theorem \ref{firstthm}. To be more precise, we will apply Theorem \ref{firstthm} to prove Theorem \ref{Cxizero}.\\
    {\bf Proof of Theorem \ref{Cxizero}.}
         The proof of the rigidity result employs a method analogous to the one used in the proof of Proposition D4 in \cite{Lafontaine1983}. Let $f$ be characteristic function of $\xi.$ Under the assumptions of Theorem \ref{Cxizero}, on the basis of Remark \ref{closedmathcalLf}, it is easy to see that the function $f$ fufills
         \[f_{ik}-\Delta f g_{ik}-f R_{ik}=-C_{kji}\xi^j.\]
         By taking trace on both sides of previous formula, one can compute that
         \[\Delta f+\frac{\rm{R}}{n-1}f=0.\]
         Combining the previous two formulas, we deduce that
         \begin{equation}\label{fvss}
           f_{ik}+  \frac{\rm{R}}{n(n-1)}fg_{ik}-f E_{ik}=-C_{ijk}\xi^j.
         \end{equation}
         On the other hand, by taking $X=\xi$ in (\ref{R}), we know that $\xi$ and $\nabla f$ are linearly dependent. Let $\mathcal{Z}$ denote the set of zero point of $\xi.$ Therefore, there exists a smooth function $\rho:M\backslash \mathcal{Z}\rightarrow\mathbb R$ such that
         \begin{equation}\label{newclosed}
                       \nabla f+\rho\xi=0,\text{~in~} M\backslash \mathcal{Z}.
         \end{equation}
         Take account of (\ref{nablaXi}) to infer
         \begin{equation}\label{nabla2f}
           \nabla^2f=-d\rho\otimes\xi^\flat-\rho fg,\text{~in~} M\backslash \mathcal{Z}.
         \end{equation}
         Put (\ref{fvss}) and (\ref{nabla2f}) together to imply
         \begin{equation}\label{fErho}
           fE_{ik}=-\rho_i\xi_k^\flat+(\frac{{\rm{R}}}{n(n-1)}-\rho)fg_{ik}+C_{ijk}\xi^j,\text{~in~} M\backslash \mathcal{Z}.
         \end{equation}
         Keeping remind of (\ref{newclosed}), we obtain
         \[fE_{ik}T_{ijk}f_{j}=-\rho fE_{ik}T_{ijk}\xi_{j}^\flat=0,\text{~in~} M\backslash \mathcal{Z},\]
         where the second equality follows from (\ref{fErho}), (\ref{Tiji=0}), the assumption $C_{ijk}\xi^j=0,$ and the fact that $(0,3)-$tensor ${\rm T}$ is skew-symmetric with respect to the last two indices. In conjunction with Lemma \ref{formforTfE} and the fact that the set $\mathcal{Z}$ is isolated and the set $\{f\neq0\}$ is dense in $M$, one has
         \[{\rm T}\equiv0,\text{~in~} M.\]
         We distinguish the following three cases.\\
         {\bf Case I.} For $n=3,$ from (\ref{fCVSS}), we see that the Cotton tensor $\rm C$ is zero on $M^3$. Hence, $(M^3 ,g)$ is locally conformally flat. According to \cite[Theorem 3.5]{Kobayashi1982} or \cite[Theorem C1]{Lafontaine1983}, we deduce that $(M^3, g)$ is isometric to either $\mathbb S^3(1)$ or a finite quotient of $\mathbb S^1(\sqrt{1/3})\times\mathbb S^2(\sqrt{1/3})$ or a finite quotient of the warped product $\mathbb S^1(\sqrt{1/3})\times_h \mathbb{S}^{2}(\sqrt{1/3}).$\\
         {\bf Case II.} For $n=4,$ following the proof of Theorem 20 in \cite{Jian2023}, we find that the vanishing of $\rm T$ implies the vanishing of the Cotton tensor $\rm C.$ Combining this with (\ref{fCVSS}), we have
         \[i_{\nabla f}W\equiv0,\text{~on~} M.\]
         Combining this with the classical identity for the Weyl tensor:$W_{ijkl}W_{pjkl}=\frac{1}{4}|\operatorname W|^{2}$ $ g_{ip}$(cf.\cite[p.413]{Derdzinski1983}), it follows that $(M^4, g)$ is locally conformally flat. From \cite[Theorem 3.5]{Kobayashi1982} or \cite[Theorem C1]{Lafontaine1983}, we deduce that $(M^4, g)$ is isometric to either $\mathbb S^4(1)$ or a finite quotient of $\mathbb S^1(1/2)\times\mathbb S^{3}(\sqrt{1/2})$ or a finite quotient of $\mathbb S^1(1/2)\times_h\mathbb S^{3}(\sqrt{1/2}).$\\
         {\bf Case III.} For $n\ge5,$ by \cite{QingYuan2013}, we know that $(M^n, g)$ is isometric to either $\mathbb S^n(1)$ or a finite quotient of $\mathbb S^1(\sqrt{1/n})\times F$ or a finite quotient of $\mathbb S^1(\sqrt{1/n})\times_h F$ where $F$ is an $(n-1)-$dimensional closed Einstein manifold with Einstein constant $n.$\\ The desired results are obtained by observing that each conformal vector field on the Riemannian product $\mathbb S^1(\sqrt{1/n})\times F $ is Killing. This completes the proof of Theorem \ref{Cxizero}.

    \begin{Rmk}
         \begin{enumerate}[leftmargin=0.5cm, itemindent=0.4cm]
           \item[{\rm(1)}.] It is worthwhile to note that the curvature assumption is automatically satisfied when $(M^n, g)$ has harmonic curvature. The associated classification result was obtained in Proposition ${\rm D4}$ of {\rm\cite{Lafontaine1983}}.
           \item[{\rm(2)}.] From the proof of Theorem {\rm\ref{Cxizero}}, it is easy to see that the $(0,3)-$tensor ${\rm T}$ still vanishes when $(M^n, g)$ is a complete connected Riemannian manifold.
         \end{enumerate}
    \end{Rmk}

           As further applications of Theorem \ref{firstthm}, we will utilize it in the subsequent section to investigate the rigidity of vacuum static spaces and related critical spaces that admit a non-Killing closed conformal vector field.

\section{Critical spaces with closed conformal vector fields}
         The main purpose of this section is to classify closed manifolds with a non-Killing closed conformal vector field that admit nonconstant solutions to (\ref{MainEab}) satisfying (\ref{ba=0}).
    \subsection{A fundamental ingredient}
         We begin with a basic lemma.
    \begin{Lem}\label{CxiCxi}
       Let $(M^n, g)$ be an $(n\ge3)-$dimensional Riemannian manifold with constant scalar curvature. If $\xi$ is a closed conformal vector field on $(M^n, g)$, then the following two assertions hold true:
       \begin{itemize}[leftmargin=0.5cm, itemindent=0.4cm]
         \item[(i).] $C_{ijk}\xi^i=0.$
         \item[(ii).] $\Xi_{ik}\xi^i=0,$ where $\Xi_{ik}:=C_{ijk,j}.$
       \end{itemize}
    \end{Lem}
    \begin{proof}
        The first item follows from Proposition \ref{iC}. The item $(ii)$ is derived by taking covariant derivative on both sides of the equation in the item $(i)$ in the direction of $\frac{\partial}{\partial x^j},$ together with (\ref{nablaXi}) and (\ref{vofC}). The proof of Lemma \ref{CxiCxi} is finished.
    \end{proof}

        Relied on Lemma \ref{CxiCxi}, we have
    \begin{Prop}\label{ldCTzero}
        Let $(M^n, g)$ be an $(n\ge3)-$dimensional connected Riemannian manifold admitting a non-Killing closed conformal vector field $\xi.$ Assume that $f$ is a nonconstant smooth solution to {\rm(\ref{MainEab})} on $(M^n, g)$. If $\xi$ and $\nabla f$ are linearly dependent on $M^n,$ then both $(f+a){\rm C}$ and the $(0,3)-$tensor ${\rm T}$ vanishes identically on $(M^n, g)$.
    \end{Prop}
    \begin{proof}
        Let $\varphi$ be the characteristic function of $\xi.$ As mentioned in subsection 2.3, the scalar curvature of $(M^n, g)$ is constant. We firstly show that the Cotton tensor is zero. To see this, according to (\ref{RxifP}) and the closedness of $\xi,$ it follows that
       \begin{equation}\label{Rlijkprop}
         R_{lijk}\xi^l=g_{ij}\varphi_k-g_{ik}\varphi_j,
       \end{equation}
       where $\varphi$ is characteristic function of $\xi.$ By contracting with $C_{ijk}$ on both sides of (\ref{Rlijkprop}), in conjunction with (\ref{vofC}), we have
       \begin{equation*}
         R_{lijk}\xi^lC_{ijk}=0.
       \end{equation*}
       On the other hand, by contracting with $C_{ijk}$ on both sides of (\ref{RlijkflVSS}), in conjunction with (\ref{vofC}) and the fact that the Cotton tensor is skew symmetric with respect to the last two indices, we get
        \begin{equation*}
         R_{lijk}f_lC_{ijk}=-2E_{ik}C_{ijk}f_{j}+(f+a)\|{\rm C}\|^2.
       \end{equation*}
       In view of previous two formulas, the assumption that $\xi$ and $\nabla f$ are linearly dependent on $M^n,$ and the fact that the zero set of $\xi$ is isolated, one can see that
       \begin{equation}\label{CECf}
         (f+a)\|{\rm C}\|^2=2E_{ik}C_{ijk}f_{j}, \text{on}, M.
       \end{equation}
       We now turn our attention to the term $E_{ik}C_{ijk}f_{j}.$ To do that, from the item $(i)$ and $(ii)$ in Lemma \ref{CxiCxi}, the assumption that $\xi$ and $\nabla f$ are linearly dependent on $M^n$ and the zero set of $\xi$ is isolated, we conclude
       \begin{equation}\label{Cfzero}
         C_{ijk}f_i=0, \text{on}, M
       \end{equation}
       and
       \begin{equation}\label{Ccoafzero}\Xi_{ik}f_i=0, \text{on}, M.\end{equation}
       Therefore, by taking covariant derivative on both sides of (\ref{Cfzero}) in the direction of $\frac{\partial}{\partial x^k},$ together with (\ref{Ccoafzero}), (\ref{MainEab}), (\ref{vofC}), and the fact that ${\rm C}$ is skew symmetric with respect to the last two indices, one has
       \begin{equation*}
         (f+a)E_{ik}C_{ijk}=0, \text{on}, M.
       \end{equation*}
       Combine previous formula  with (\ref{CECf}) to imply
       \begin{equation*}
         (f+a)\|{\rm C}\|^2=0, \text{on}, M.
       \end{equation*}
       We next show that ${\rm T}$ also vanishes identically on $M.$ To see this, by contracting with $T_{ijk}$ on both sides of (\ref{Rlijkprop}) and (\ref{RlijkflVSS}), in conjunction with (\ref{Tiji=0}), the assumption that $\xi$ and $\nabla f$ are linearly dependent on $M^n,$ the fact that the zero set of $\xi$ is isolated, the first equality in (\ref{Tijk=-Tikj}), and the vanishing of $(f+a){\rm C},$ we obtain
       \[0=-2E_{ik}T_{ijk}f_{j}, \text{on}, M.\]
       The vanishing of ${\rm T}$ is deduced from combining previous formula with Lemma \ref{formforTfE}. The proof of Proposition \ref{ldCTzero} is completed.
    \end{proof}
    \begin{Rmk}\label{nlin}
       \begin{itemize}[leftmargin=0.5cm, itemindent=0.4cm]
         \item[{\rm(1)}.] Proposition {\rm\ref{ldCTzero}} does not depend on the sign of the scalar curvature.
         \item[{\rm(2)}.] The conclusion holds provided that the characteristic function $\varphi$ satisfies {\rm(\ref{MainEVSS})}.
         \item[{\rm(3)}.] In general, the vectors $\xi$ and $\nabla f$ may be linearly independent. For instance, consider the unit sphere $\mathbb{S}^n(1)=\{x=(x_1,\cdots,x_{n+1})\in\mathbb{R}^{n+1}:\sum_{i=1}^{n+1}x_i^2=1\},$ and let $k\neq l\in\{1,\cdots,n+1\}.$ Assume that $\xi=\nabla x_k$ and $f=x_l+b.$ It is well-known that $\xi$ is a gradient conformal vector field and $f$ satisfies {\rm(\ref{MainEVSS})} on $\mathbb{S}^n(1).$ However, $\xi$ and $\nabla f$ are linearly independent in the set $\{x\in\mathbb{S}^n(1):x_k^2+x_l^2<1\}.$
       \end{itemize}
    \end{Rmk}
  \subsection{Critical spaces with closed conformal vector fields}

    \begin{Lem}\label{xiCVFtwoformulas}
        Let $(M^n, g)$ be an $(n\ge3)-$dimensional connected Riemannian manifold and $\xi$ be a conformal vector field on $(M^n, g)$ with characteristic function $\varphi.$ Assume that $f$ is a nonconstant smooth solution to {\rm(\ref{MainEab})} on $(M^n, g)$. Then the following two assertions hold true:
        \begin{itemize}[leftmargin=0.5cm, itemindent=0.4cm]
          \item[{\rm (1)}.] $n(f_j\varphi_k-f_k\varphi_j)+\frac{\operatorname R}{n-1}(f_j\xi_k^{\flat}-f_k\xi_j^{\flat})=\mathcal{P}_{jk,i}f_i+f_j\mathcal{P}_{ik,i}-f_k\mathcal{P}_{ij,i}-(f+a)C_{ijk}\xi^i.$
          \item[{\rm (2)}.] $\xi(f)E_{ik}-\langle\nabla\varphi+\frac{\operatorname R}{n(n-1)}\xi, \nabla f\rangle g_{ik}=-f_k(n\varphi_i+\frac{\operatorname R}{n-1}\xi_i^{\flat})+f_j\mathcal{P}_{ji,k}+f_k\mathcal{P}_{li,l}+(f+a)C_{ijk}\xi^j.$
        \end{itemize}
    \end{Lem}
    \begin{proof}
        For the item $(1),$ multiplying both sides of (\ref{RxifP}) by $f_i$ and summing up over $i$ from $1$ to $n,$ we deduce
        \begin{equation*}
             R_{lijk}\xi^lf_i=f_j\varphi_k-f_k\varphi_j-\mathcal{P}_{jk,i}f_i.
       \end{equation*}
       On the other hand, multiplying both sides of (\ref{RlijkflVSS}) by $\xi^i$ and summing up over $i$ from $1$ to $n,$ it follows that
       \begin{eqnarray*}
       R_{lijk}f_{l}\xi^i&=&f_{k}\big(R_{ij}\xi^i-\frac{\operatorname R}{n-1}\xi_j^{\flat}\big)-f_{j}\big(R_{ik}\xi^i-\frac{\operatorname R}{n-1}\xi_k^{\flat}\big)+(f+a)C_{ijk}\xi^i \\
          &=& (n-1)(f_j\varphi_k-f_k\varphi_j)+\frac{\operatorname R}{n-1}(f_j\xi_k^{\flat}-f_k\xi_j^{\flat})\\
          &&+f_k\mathcal{P}_{ij,i}-f_j\mathcal{P}_{ik,i}+(f+a)C_{ijk}\xi^i,
       \end{eqnarray*}
       where the second equality follow from (\ref{divofP}). The desired equality is achieved by gathering the previous two formulas together.\\
       For the item $(2),$ multiplying both sides of (\ref{RxifP}) by $f_j$ and summing up over $j$ from $1$ to $n,$ one can see that
       \begin{equation*}
             R_{lijk}\xi^lf_j=f_i\varphi_k-\langle\nabla f,\nabla\varphi\rangle g_{ik}-f_j\mathcal{P}_{jk,i}.
       \end{equation*}
       On the other hand, multiplying both sides of (\ref{RlijkflVSS}) by $\xi^j$ and summing up over $j$ from $1$ to $n,$ we arrive at
       \begin{eqnarray*}
       R_{lijk}f_{l}\xi^j&=&f_{k}\big(R_{ij}\xi^j-\frac{\operatorname R}{n-1}\xi_i^{\flat}\big)-\xi(f)\big(E_{ik}-\frac{\operatorname R}{n(n-1)}g_{ik}\big)+(f+a)C_{ijk}\xi^j \\
          &=& -(n-1)f_k\varphi_i-\frac{\operatorname R}{n-1}f_k\xi_i^{\flat}-\xi(f)\big(E_{ik}-\frac{\operatorname R}{n(n-1)}g_{ik}\big)\\
          &&+f_k\mathcal{P}_{li,l}+(f+a)C_{ijk}\xi^j,
       \end{eqnarray*}
       where the second equality follow from (\ref{divofP}). The desired equality is derived by combining the previous two formulas.
    \end{proof}

       Based on Theorem \ref{Cxizero}, Proposition \ref{ldCTzero}, and Lemma \ref{xiCVFtwoformulas}, we have
    \begin{Thm}\label{CSS}
        Let $(M^n, g)$ be an $(n\ge3)-$dimensional closed Riemannian manifold with scalar curvature $n(n-1)$ that admits a non-Killing closed conformal vector field. Suppose that $f$ is a nonconstant smooth solution to (\ref{MainEab}) on $(M^n, g)$ with $a$ and $b$ satisfying
          \begin{equation}\label{ba=0}
            b+\frac{\operatorname R}{n(n-1)}a\neq0.
          \end{equation}
        Then $(M^n, g)$ is isometric to one of the spaces listed in Theorem {\rm\ref{Cxizero}}.
    \end{Thm}
    \begin{proof}   
        Let $\varphi$ be characteristic function of a non-Killing closed conformal vector field $\xi$ on $(M^n, g)$. Let $\kappa:=b+\frac{{\rm R}}{n(n-1)}a\neq0.$ Thus $\kappa$ can be either positive or negative. Note that if the triple $(f, a, b)$ satisfies (\ref{MainEab}), then $(-f,-a,-b)$ also fulfills (\ref{MainEab}). Therefore, without loss of generality, we can assume that $\kappa>0.$ With the aid of Lemma \ref{scalarPosi}, we know that the scalar curvature ${\rm R}>0.$ Combining this with (\ref{Inu}) and the positivity of $\kappa,$ we can deduce that the set \(\{f+a>0\}\) is nonempty. The proof is divided into the following three steps.\\
    {\bf Step 1.}  In the first step, if the level set $f^{-1}(-a)$ is empty, we prove that $(M^n, g)$ is isometric to a round sphere with radius $\sqrt{\frac{n(n-1)}{\operatorname R}}.$\\
        Assume that $f^{-1}(-a)$ is empty. Thus the quantity $f+a$ does not change sign. Without loss of generality, we suppose that $f+a$ is positive on $M.$ By integrating both sides of the divergence formula
        $$\text{div}({\rm E}(\nabla f))=(f+a)\|{\rm E}\|^2$$
        over $M,$ one has
        \[\int_M(f+a)\|{\rm E}\|^2=0.\]
        Therefore, $\|{\rm E}\|^2=0$ on $M.$ Hence, the equation (\ref{MainEab}) becomes
        \[\nabla^2f+\frac{{\rm{R}}}{n(n-1)}fg=bg.\]
        Equivalently,
        \[\nabla^2\tilde{f}+\frac{{\rm{R}}}{n(n-1)}\tilde{f}g=0,\]
        where $\tilde{f}:=f-\frac{n(n-1)}{{\rm{R}}}b.$ The desired result is obtained by combining this with Obata's theorem \cite[Theorem A]{Obata1962}.\\
    {\bf Step 2.} In the second step, if $f^{-1}(-a)$ is nonempty, we show that the level set $f^{-1}(-a)$ has measure zero.\\
        To see this, we rewrite (\ref{MainEab}) in terms of $f+a$ as follows
           \[\nabla^2f=(f+a)\left({\rm E}-\frac{{\rm R}}{n(n-1)}g\right)+\kappa g.\]
        Together with the positivity of $\kappa,$ we conclude that each critical point of $f$ on $f^{-1}(-a)$ is non-degenerate. Therefore, the level set $f^{-1}(-a)$ has measure zero.\\
    {\bf Step 3.} In third step, we verify that the Cotton tensor vanishes identically. \\
        To see this, thanks to the fact that the level set $f^{-1}(-a)$ has measure zero, it suffices to show that the quantity $(f+a){\rm C}$ is equal to zero on $M.$ We will argue by contradiction to show that $(f+a){\rm C}\equiv0$. Suppose that $(f+a){\rm C}$ does not vanish at each point on a connected open subset $\Omega$ of $M.$ Thus, by (\ref{fCVSS}) and definition of $\operatorname{T},$ it is clear that $|\nabla f|^2>0$ at each point in $\Omega.$ By means of Proposition \ref{ldCTzero}, it is easy to see that $\xi$ and $\nabla f$ are linearly independent on $\Omega\backslash \mathcal{Z}$ where $\mathcal{Z}=\{\xi=0\}.$

       We next utilize Proposition \ref{ldCTzero} and Lemma \ref{xiCVFtwoformulas} to derive a contradiction. Firstly, by virtue of the item $(1)$ in Lemma \ref{xiCVFtwoformulas}, the closedness of $\xi,$ the item $(i)$ in Lemma \ref{CxiCxi}, and (\ref{newclosed}), we immediately get
       \[(\rho-\frac{\operatorname R}{n(n-1)})(f_j\xi_k^{\flat}-f_k\xi_j^{\flat})=0,\text{~in~} M\backslash \mathcal{Z}.\]
       The fact that $\xi$ and $\nabla f$ are linearly independent on $\Omega\backslash \mathcal{Z}$ yields
       \begin{equation}\label{rhov}
         \rho-\frac{\operatorname R}{n(n-1)}=0,\text{~in~} \Omega\backslash \mathcal{Z}.
       \end{equation}
       As a result, making use of (\ref{ricandf}), (\ref{rhov}), and (\ref{newclosed}), we have
       \begin{equation}\label{Exizero}
         E_{ik}\xi^k=0,\text{~in~} \Omega\backslash \mathcal{Z}.
       \end{equation}
       Secondly, in view of the item $(2)$ in Lemma \ref{xiCVFtwoformulas}, the closedness of $\xi,$ (\ref{rhov}), and (\ref{newclosed}), it can be inferred that
       \[\xi(f)E_{ik}=(f+a)C_{ijk}\xi^j,\text{~in~} \Omega\backslash \mathcal{Z}.\]
       We now deal with the term $C_{ijk}\xi^j.$ Based on (\ref{fErho}), (\ref{rhov}), and the constancy of scalar curvature, it follows that
       \begin{equation*}\label{varphiECxi}\varphi E_{ik}=C_{ijk}\xi^j,\text{~in~} \Omega\backslash \mathcal{Z}.\end{equation*}
       Combine previous two formulas to imply
       \begin{equation*}\label{xiffvarphi}
       \left(\xi(f)-(f+a)\varphi\right)E_{ik}=0,\text{~in~} \Omega\backslash \mathcal{Z}.
       \end{equation*}
       Observe that $\|{\rm E}\|^2$ can not equal zero at each point in $\Omega\backslash \mathcal{Z},$ otherwise ${\rm C}$ would be zero at that point. Thus there exists an open subset $\widetilde{\Omega}\subset \Omega\backslash \mathcal{Z}$ such that
       \begin{equation*}\label{xiffvarphi}
         \xi(f)=(f+a)\varphi,\text{~in~} \widetilde{\Omega}.
       \end{equation*}
       By taking covariant derivative on both sides of preceding formula in the direction of $\xi,$ together with (\ref{nablaXi}), (\ref{MainEab}), and (\ref{Exizero}), a direct computation reveals that
       \[\Big(b-\frac{\operatorname R}{n(n-1)}f\Big)\|\xi\|^2=(f+a)\xi(\varphi),\text{~in~} \widetilde{\Omega}.\]
       Combining this with (\ref{newclosed}) and (\ref{rhov}), one has
       \[\Big(b+\frac{\operatorname R}{n(n-1)}a\Big)\|\xi\|^2=0,\text{~in~} \widetilde{\Omega}.\]
       Taking advantage of the assumption (\ref{ba=0}), we find that $\xi$ is zero at each point in $\Omega\backslash \mathcal{Z}.$ This contradicts the fact that $\xi$ is not zero in $\Omega\backslash \mathcal{Z}.$ Eventually, the quantity $(f+a){\rm C}$ is equal to zero on $M.$ Thanks to $(2)$ in Remark \ref{closedmathcalLf}, the triple $(M^n, g, \varphi)$ is a vacuum static space. Therefore, the desired rigidity result is obtained by Theorem \ref{Cxizero}. This ends the proof of Theorem \ref{CSS}.
    \end{proof}
    \begin{Rmk}
        It is easy to see that $f+a$ satisfies {\rm(\ref{MainEVSS})} when $b+\frac{\operatorname R}{n(n-1)}a=0$ and the above process fails to yield the contradiction. We will employ an alternative approach to derive the contradiction in the subsequent section.
    \end{Rmk}

\section{Vacuum static spaces with closed conformal vector fields}

       On the basis of Theorem \ref{firstthm}, Proposition \ref{ldCTzero}, and Lemma \ref{xiCVFtwoformulas}, we have
    \begin{Thm}\label{VSSCVF}
       Let $(M^n, g, f)$ be an $(n\ge3)-$dimensional closed vacuum static space with scalar curvature $n(n-1).$ Then $(M^n, g)$ admits a non-Killing closed conformal vector field if and only if it is isometric to one of the spaces listed in Theorem {\rm\ref{Cxizero}}.
    \end{Thm}
    \begin{proof}
       Let $\varphi$ be characteristic function of a non-Killing closed conformal vector field $\xi$ on $(M, g)$. Similar to the proof of Theorem \ref{Cxizero}, it is sufficient to demonstrate that ${\rm T}\equiv 0$ on $M.$ \\
       For $n=3,$ by similar approach to the proof of Theorem \ref{CSS}, one can derive the vanishing of ${\rm T}.$ To see this, using the notation from the proof of Theorem \ref{CSS}, we multiply both sides of (\ref{varphiECxi}) by $f$ and $\xi^i,$ and summing up over $i$ from $1$ to $n,$ thanks to (\ref{Exizero}), one has
       \[0=\varphi fE_{ik}\xi^i=fC_{ijk}\xi^i\xi^j=T_{ijk}\xi^i\xi^j,\text{~in~} \Omega\backslash \mathcal{Z},\]
       where the second equality follows from (\ref{fCVSS}) and the fact that the Weyl tensor equals zero when the dimension $n=3.$ In conjunction with the definition of ${\rm T}$ and (\ref{Exizero}), one can compute that
       \[\|\xi\|^2E_{kl}f_l=0,\text{~in~} \Omega\backslash \mathcal{Z},\]
       Consequently,
       \[E_{kl}f_l=0,\text{~in~} \Omega\backslash \mathcal{Z}.\]
       By taking covariant derivative on both sides of previous formula in the direction of $\frac{\partial}{\partial x^k},$ according to the relation $E_{kl,k}=\frac{n-2}{2n}\nabla_l{\rm R},$ the constancy of scalar curvature, and (\ref{MainEVSS}), we have
       \[f\|{\rm E}\|^2=0,\text{~in~} \Omega\backslash \mathcal{Z}.\]
       Keeping the fact that the set $\{f\neq0\}$ is dense in $M$ and the definition of ${\rm T}$ in mind, we conclude that ${\rm T}\equiv0$ in $\Omega\backslash \mathcal{Z}.$ This leads to a contradiction. Therefore, the $(0,3)-$tensor ${\rm T}$ must be zero on the entire manifold $M^n.$\\
       For $n\geq4,$ we will exploit another method to show that ${\rm T}\equiv 0$ on $M.$ To do that, from the proof of Theorem {\rm\ref{CSS}}, one obtains
       \[\|{\rm T}\|^2\|\nabla\varphi+\frac{\operatorname R}{n(n-1)}\xi\|^2=0,\text{~in~} M.\]
       If $\nabla\varphi+\frac{\operatorname R}{n(n-1)}\xi\equiv0$ on $M,$ a direct calculation indicates that
       \[\nabla^2\varphi+\frac{\operatorname R}{n(n-1)}\varphi g=0,~\text{on}~ M.\]
       In view of Obata's theorem \cite[Theorem A]{Obata1962}, we deduce that $(M^n, g)$ is isometric to a round sphere with radius $\sqrt{\frac{n(n-1)}{\operatorname R}}.$ \\
       If $\nabla\varphi+\frac{\operatorname R}{n(n-1)}\xi\neq0$ on an open subset $V\subset M$, it is easy to see that ${\rm T}\equiv0$ on $V.$ By analyticity, we know that ${\rm T}\equiv0$ on $M.$ The desired result is obtained as in the proof of Theorem \ref{Cxizero}. This completes the proof of Theorem \ref{VSSCVF}.
    \end{proof}

   \begin{Rmk}
       It is worth noting that ${\rm T}$ may not equal zero on a completed connected vacuum static space with negative scalar curvature that carrying a non-Killing closed conformal vector field. For instance, let $(M^n, g)$ be the warped product in Example ${\rm 1}$, as mentioned at the end of section ${\rm 3}$, it is a vacuum static space with scalar curvature $-n(n-1)$ that carrying a non-Killing closed conformal vector field $h\frac{\partial}{\partial t}.$ However, its Ricci curvature tensor has three distinct eigenvalues. Therefore, the associated $(0,3)-$tensor ${\rm T}$ is also non identically zero.
   \end{Rmk}

       We end this article with three interesting questions.
   \begin{Que}
       Does Theorem {\rm\ref{VSSCVF}} still hold when the closed conformal vector field is replaced by a general conformal vector field?
   \end{Que}
   \begin{Que}
       Is it possible to classify complete connected vacuum static space with a non-Killing conformal vector field?
   \end{Que}

 \section{Acknowledgement}
        We would like to express our gratitude to the referee for careful reading of this manuscript.

\end{document}